\documentclass[11pt]{amsart} 
\usepackage{graphicx}
\usepackage{amsfonts}
\usepackage{amsmath}
\usepackage{amssymb} 
\usepackage{amsthm}
\usepackage{amscd}
\usepackage[all]{xy}

\newtheorem{definition}{Definition}[section]
\newtheorem{theorem}[definition]{Theorem}

\newtheorem{corollary}[definition]{Corollary}

\newtheorem{lemma}[definition]{Lemma}

\newtheorem{proposition}[definition]{Proposition}

\newtheorem{remark}[definition]{Remark}

\newtheorem*{theorem*}{Theorem}
\newtheorem*{corollary*}{Corollary}

\renewcommand{\epsilon}{\varepsilon}
\renewcommand{\deg}{\mbox{deg}}

\newcommand{\larr}{\longrightarrow}
\newcommand{\arr}{\rightarrow}
\newcommand{\corr}{\mbox{Corr}}

\newcommand{\calm}{\mathcal{M}}

\newcommand{\psd}{{\mathbb{P}}^{d-1}}

\newcommand{\rch}{CH_{\mathbb{Q}}}
\newcommand{\kun}{K\"{u}nneth}
\newcommand{\kuns}{K\"{u}nneth }
\newcommand{\Q}{\mathbb{Q}}

\newcommand{\kunn}{K\"{u}nnemann }

\numberwithin{equation}{subsection}

\def\incl{\hookrightarrow}

\def\directsum{\bigoplus}

\def\zah{{\mathbb{Z}}}

\def\ratls{{\mathbb{Q}}}
\def\spec{\mbox{Spec~}}

\def\nin{\not \in}

\def\ker{\mbox{Ker }}
\def\im{\mbox{Im }}

\def\endproof{\medskip}

\def\ds{\displaystyle}

\def\pic0{{\mbox{Pic}}^0}

\def\p1k{{\mathbf{P}}^1_k}

\begin{document}

\baselineskip=16.5pt
\pagenumbering{arabic}
\thispagestyle{empty}

\title{Explicit Chow-Lefschetz decompositions for Kummer manifolds}
\author{Reza Akhtar}
\address{Department of Mathematics, Miami University, Oxford, Ohio, 45056, USA.}
\email{akhtarr@miamioh.edu}
\thanks{  }  
\author{Roy Joshua}
\address{Department of Mathematics, Ohio State University, Columbus, Ohio,
43210, USA.}
\email{joshua@math.ohio-state.edu}
\thanks{ The second author was supported by the NSF}


\begin{abstract}  
Let $X$ be the quotient of a smooth projective variety over a field by a finite group action (in which case we say $X$ is pseudo-smooth), such that the singularities of $X$ are isolated $k$-rational points. Let 
$f: Y \arr X$ be the morphism obtained by blowing up these 
points on $X$.  Assume further that $Y$ is pseudo-smooth, and
that the components of the exceptional divisor are projective spaces.
Then, without invoking the theory of finite-dimensional motives
or assuming any of the standard conjectures, we show  that a Chow-\kuns decomposition on
either $X$ or $Y$ gives rise, by means of an explicit construction, to
a Chow-\kuns decomposition on the other.  We use these constructions to
show that various properties (among them Murre's Conjectures and being of
Lefschetz type) hold for $X$ if and only if they hold for $Y$.  The main examples of interest to us are {\em Kummer manifolds}: these are
obtained by taking the quotient of an abelian variety by the involution $a \mapsto -a$, 
and then blowing up the singular locus.  We give several further
applications of our construction to this particular class of examples.

\end{abstract}
\maketitle
\medskip

\section{\bf Introduction and summary of results}

Let $X \mapsto H^*(X)$ be a Weil cohomology theory on varieties over some algebraically closed field.  According to the standard conjectures of Grothendieck formulated in \cite{gr}, one expects -- among other things -- that if $X$ has dimension $d$, then the \kuns components of the diagonal class $[\Delta_X] \in H^d(X \times X)$ should lie in the subgroup $A^d(X \times X)$ of $H^{\dim X}(X \times X)$ generated by algebraic cycles.  Moreover, for any smooth hyperplane section $W \subseteq X$, the so-called Hard Lefschetz Theorem should hold: specifically, if $L: H^i(X) \arr H^{i+2}(X)$ is the Lefschetz operator, then composing with the iterated operator $L^{d-i}$ should define an isomorphism $A^i(X) \arr A^{d-i}(X)$ for all $i$.  A detailed exposition of the standard conjectures may be found in the comprehensive survey article of Kleiman \cite{kl}.  
  
\medskip

A related but stronger set of conjectures was formulated by Jacob Murre in \cite{mur}.  Murre conjectured that the \kuns components of the 
diagonal class of each such $X$ should actually be defined in the category of (rational) Chow motives, and that these projectors should act on the 
(rational) Chow groups of $X$ in a prescribed manner.  The first of Murre's conjectures -- the existence of a Chow-\kuns decomposition -- 
has been verified for certain classes of varieties (curves, surfaces, abelian varieties, and various other special cases); however, it remains 
wide open in the general case.  The existence of a Chow-\kuns decomposition for abelian varieties was first demonstrated by Shermenev \cite{sh}, 
although it is a later construction of the same by Deninger and Murre \cite{dm} that lends itself most readily to applications.  \kunn \cite{ku} used 
the Deninger-Murre construction to prove that the Hard Lefschetz Theorem holds for abelian varieties at the level of Chow motives.  In fact, \kunn 
proved much more, constructing Lefschetz, Lambda, and *-operators for abelian varieties, and showing that various identities among these, which hold in the setting of 
K\"{a}hler geometry, actually hold at the level of Chow groups.  More significantly, he showed that if a variety is of {\em Lefschetz type} (see Section \ref{lefalgebra}), then many expected properties -- including the Hard Lefschetz Theorem and the existence of projectors appropriately refining the Chow-\kuns decomposition -- follow immediately.  

\medskip

Let $A$ denote an abelian variety over an algebraically closed field of characteristic different from $2$.  Its associated {\em Kummer variety} $K_A$ is the quotient of $A$ by the involution $a \mapsto -a$.  If $A$ has dimension $d>0$, then $K_A$ has $2^{2d}$ singular points, which are precisely the images of the $2$-torsion points of $A$ under the quotient map $A \arr K_A$.  Blowing up these points yields a smooth variety $K_A'$ which we call the {\em Kummer manifold}.  Even though $K_A$ is a singular variety, it is pseudo-smooth (i.e. the quotient of a smooth variety by the action of a finite group scheme), so basic methods of intersection theory may be used to study its Chow groups with $\ratls$-coefficients (see \cite[Example 1.7.6]{fu}).  In earlier work \cite{aj1}, the authors of the present article used the  Chow-\kuns decomposition for $A$ constructed by Deninger and Murre to construct an explicit  Chow-\kuns decomposition for $K_A$.  Although the {\em existence} of such a decomposition follows from the theory of finite-dimensional motives (see \cite{gp}), we gave several applications for our construction which would not have been possible from the abstract theory.  This work was continued in \cite{aj2}, in which we used K\"{u}nnemann's Lefschetz algebra structure on the Chow groups of an abelian variety to establish one for Kummer varieties.  Once again, the {\em existence} of such a decomposition was established in \cite{kmp} under the assumption of parts of the standard conjectures (which are known to hold for Kummer manifolds in characteristic $0$ by work of Arapura \cite{ar}); however, our construction has no dependence on characteristic and furthermore lends itself to several applications.

\medskip

In the present article, we use our constructions for $K_A$ to exhibit an explicit Chow-\kuns decomposition for the Kummer manifold $K_A'$ and also an 
explicit Lefschetz algebra structure on its Chow groups.  The following is a technical result, which combined with our earlier
results (see \cite{aj1} and \cite{aj2}), provides the Chow-\kuns decomposition for the Kummer manifolds.
  
 \begin{theorem*}\label{summaryth}(See Theorems \ref{blowup} and \ref{leftheorem})
 Let $X$ denote a pseudo-smooth variety of dimension $d$ over a field $k$ and $Y$ the variety obtained by blowing up a finite number of $k$-rational points on $X$.  
Suppose further that $Y$ is pseudo-smooth, and that the (respective) exceptional divisors of the blow-up at each point are isomorphic to ${\mathbb{P}}^{d-1}$.  
\vskip .2cm 

Then:

 \begin{itemize}
\item 
If either $X$ or $Y$ has a Chow-\kuns decomposition, then this can be used to construct (explicitly) a Chow-\kuns decomposition on the other (cf. \eqref{CK.for.Y} and Corollary \ref{blowdown}.)

\item
If the Chow-\kuns decomposition (so constructed) on either $X$ or $Y$ satisfies Poincar\'e duality (respectively, Murre's Conjecture {\bf B}, {\bf B'},  {\bf C}, {\bf D}) then the same is true for the other variety.  
 
\item
$Y$ is of Lefschetz type if and only if $X$ is of Lefschetz type.
 
\end{itemize}  
 
\end{theorem*}
 
When combined with the results of \cite{aj1} and \cite{aj2}, we may then conclude:

\begin{corollary*}
Let $A$ denote an abelian variety of dimension $d>0$ over an algebraically closed field of characteristic 
different from $2$ and $K_A'$ its Kummer manifold.  Then $K_A'$ has a Chow-\kuns decomposition satisfying Poincar\'{e} duality and Murre's conjecture {\bf B'}; moreover, $K_A'$ has Lefschetz type in the sense of Definition ~\ref{lefalg}.  If $d \leq 4$, then $K_A'$ also satisfies Murre's conjecture {\bf B}.  
\end{corollary*}
    
We also give several other applications of our construction.  The first  concerns powers of the relation of algebraic equivalence (on algebraic cycles).  For a pseudo-smooth variety $V$, let $L\rch^p(V)$ denote the subgroup of $\rch^p(V)$ consisting of cycles algebraically equivalent to zero.  For $r \geq 1$, we denote by $L^{*r}$ the $r$th power of (the equivalence relation) $L$, as defined by Hiroshi Saito \cite{sai}.  

\begin{theorem*}(See Proposition \ref{adeqkummer})
Let $A$ denote an abelian variety of dimension $d>0$ over an algebraically closed field of characteristic different from $2$. Let $X=K_A$ and $Y=K_A'$.  
Let $[\Delta_X] = \sum_{i=0}^{2d} \pi_i^X $ denote the Chow-\kuns decomposition for $X$ constructed in \cite{aj1} and $[\Delta_Y]=\sum_{i=0}^{2d} \pi_i^Y$ the 
Chow-\kuns decomposition for $Y$ constructed in  ~\eqref{CK.for.Y}.  Define filtrations on $\rch^*(X)$ and $\rch^*(Y)$ by $\ds F^r \rch^*(X) = \sum_{i=0}^{2d-r}  \pi_i^X \bullet \rch^p(X) $ and $\ds F^r \rch^*(Y) = \sum_{i=0}^{2d-r}  \pi_i^Y \bullet \rch^p(Y)$.  Then: 

\begin{enumerate}

\item For  $r \geq 1, ~ ~ F^r \rch^d(X) = L^{*r} \rch^d(X) \mbox{ and } F^r \rch^d(Y) = L^{*r} \rch^d(Y).$

\item For $r > d$, $L^{*r} \rch^*(X)=0$ and $L^{*r} \rch^*(Y)=0$.  

\end{enumerate} 

\end{theorem*}
 
As another application, we prove a Hard Lefschetz Theorem for Chow groups of Kummer manifolds in the case that the base field is the algebraic closure of a finite field of characteristic different from $2$.  

\begin{theorem*} (See Corollary \ref{hardlefth})
Let $Y$ denote the Kummer manifold associated to an abelian variety of dimension $d>0$ over an algebraic closure of a finite field of characteristic different from $2$, and let $L_Y$ denote the Lefschetz operator as constructed in the present paper.  Then for $2p \leq d$, the map $\rch^p(Y) \arr \rch^{d-p}(Y)$ defined by $c \mapsto L_Y^{d-2p} \bullet c$ is an isomorphism.
\end{theorem*}

\medskip

Most of our arguments rely on the following fundamental fact about the structure of the Chow groups of blow-ups of the sort we are considering.  Suppose $f: Y \arr X$ is the morphism describing the blow-up of a pseudo-smooth variety $X$ at a point, such that $Y$ is pseudo-smooth
and the exceptional divisor $Z$ is isomorphic to the projective space
over the ground field.  This is a strong assumption, but it guarantees that the cohomology of the exceptional divisors are generated by algebraic cycles,
 which may be viewed as the underlying reason why our strategy works.  In this case, $\rch^*(Y \times Y)$ is the internal
direct sum of two subgroups, which we call $A$ and $B$: $A$ consists
of cycles pulled back from $X \times X$ via $f \times f$, while $B$
consists of cycles supported on $Z \times Y \cup Y \times Z$.  With respect to the non-commutative ring structure given by the
composition of correspondences on $\rch^*(Y \times Y)=CH^*(Y \times Y) \otimes \ratls$, $A$
is a ring and $B$ is a two-sided ideal of $\rch^*(Y \times Y)$;
furthermore, $A$ and $B$ are nearly orthogonal to each other.  Using
these properties --and the crucial fact that every cycle in $B$ can be
written as a sum of external products of cycles on $Y$ -- we can,
starting with a Chow-\kuns decomposition on $X$, construct
Chow-\kuns projectors on $Y$, and then use these to construct the
appropriate operators necessary for the exhibition of a Lefschetz
algebra structure on $\rch^*(Y \times Y)$.  We also show that if a Chow-\kuns
decomposition or Lefschetz algebra structure is known for $Y$, then 
pushing forward all the relevant cycles to $X$ will establish the analogous results there.

\medskip

The paper is organized as follows.  We begin in Section \ref{preliminaries} by providing definitions and results from intersection theory.  
In Section \ref{ck}, we establish some important structural results concerning the Chow groups of a blow-up, and then give our main construction  
involving Chow-\kuns decompositions.  The remainder of the paper is devoted to various applications of our explicit constructions.
The first application, to Murre's Conjectures, appears in Section \ref{app1}.  In Section \ref{lefschetz}, we study Lefschetz decompositions in 
the context of blow-ups. Both of these are discussed in a fairly general setting. We conclude in section \ref{app2} with various specialized applications to
Kummer varieties and manifolds.
\medskip 

{\bf Acknowledgments.} 
The present work was prompted by helpful exchanges with several colleagues; in particular, we thank  Donu Arapura, Michel Brion, Igor Dolgachev, and Charles Vial for helpful discussions.  We thank the referee for a careful reading, and for several suggestions which helped improve this article. 

\medskip

\section{\bf Preliminaries}\label{preliminaries}
\subsection{Correspondences and Murre's Conjectures}\label{prelim1}
\label{basic.terminology}
\vskip .3cm

Let $k$ denote a field.  For convenience, we refer to the quotient of a smooth 
variety by the action of a finite group (scheme) as a {\em pseudo-smooth 
variety}.  It is well-known  that 
the basic machinery of intersection theory and the usual formalism for 
correspondences extends naturally from smooth varieties to pseudo-smooth 
varieties, provided one uses rational coefficients.  We may thus define the 
category $\calm_k(\Q)$ of (rational) Chow motives of 
pseudo-smooth projective varieties in the same way as for smooth projective 
varieties (see for example,  \cite{sch}).
Throughout this article, we use the notation $CH^i(X)$ for the Chow
groups of (an algebraic scheme) $X$ and write $\rch^i(X)=CH^i(X)
\otimes \Q$.  It is worth noting that if a finite group $G$ acts on a
smooth variety $X$,  then the machinery of equivariant intersection theory
allows us to identify the equivariant Chow groups
$CH^*_G(X)_{\ratls}$ with $\rch^*(X/G)$. Thus, the extension of the
usual formalism of correspondences to pseudo-smooth varieties (see \cite[Example 1.7.6]{fu}) can also
be derived from the analogous theory in the equivariant context.

\medskip
 
Since we will make use of many projection maps in the sequel, we reserve the symbol $p$ for these, with the superscript indicating the domain
 and the subscript the range.  For example, if $k$ is a field and $X,Y,Z$ are pseudo-smooth varieties over $k$, $p_{13}^{XYZ}: X \times Y \times Z \arr X \times Z$ is the map $(x,y,z) \mapsto (x,z)$. 
A subscript of $\emptyset$ 
indicates the structure morphism; for example, $p^{XY}_{\emptyset}$ is the 
structure morphism $X \times Y \arr \spec k$.  Given cycles $\alpha \in CH^i(X)$ 
and $\beta \in CH^j(Y)$, we refer to their exterior product $\alpha \times \beta = {p^{XY}_1}^* \alpha \cdot {p^{XY}_2}^* \beta$ as a {\em 
product cycle} on $X \times Y$ of {\em type} $(i,j)$; by abuse of terminology, we sometimes also refer to linear combinations of such elements as product 
cycles.  

\medskip

\medskip

Now suppose $X$, $Y$, and $Z$ are pseudo-smooth varieties over $k$,
with $\gamma \in CH^*(X \times Y)$ and $\delta \in CH^*(Y \times Z)$.
The composition $\delta \bullet \gamma \in CH^*(X \times
Z)$ is defined by $$\delta \bullet \gamma =
{p^{XYZ}_{13}}_*({p^{XYZ}_{12}}^* \gamma \cdot {p^{XYZ}_{23}}^* \delta).$$    

Composition of correspondences is associative; we will use this fact freely without explicit mention in the sequel.  
If $s: X \times Y \arr Y \times X$ is the 
exchange of factors, we define the {\em transpose} of $\alpha \in
CH^*(X \times Y)$ by $\alpha^t: = s^*(\alpha)$.  
We write $\Delta_X$ for the diagonal in $X \times X$ and $\Gamma_f$
for the graph of a morphism $f$ between (pseudo-smooth) varieties.  Since $[\Delta_X] \bullet \gamma  = \gamma= \gamma\bullet [\Delta_X]$ for $\gamma \in CH^*(X \times X)$, the operation $\bullet $ makes $CH^*(X \times X)$ into a (noncommutative) ring with unit element $[\Delta_X]$; furthermore, $CH^{\dim X}(X \times X)$ is a subring of $CH^*(X \times X)$.

\medskip

\vskip .3cm


We say that a variety $X$ of dimension $d$ has a {\em Chow-\kuns decomposition} (or CK-decomposition for short) 
if the diagonal class $[\Delta_X] \in \rch^d(X \times X)$ has a decomposition into mutually orthogonal idempotents, each of which maps onto the appropriate \kuns component under the cycle map.  More precisely, there exist $\pi_i \in 
\rch^d(X \times X)$, $0 \leq i \leq 2d$, such that:

(i) $\ds [\Delta_X]= \sum_{i=0}^{2d} \pi_i$;

(ii) $\pi_i \bullet \pi_i=\pi_i$ for all $i$, and $\pi_i \bullet \pi_j = 0$ for $i \neq j$;

(iii) If $H^*$ is a Weil cohomology theory, then for each $i$, the image of $\pi_i$ under the cycle map $cl_X: \rch^d(X \times X) 
\arr H^{2d}(X \times X; \ratls)$ is the $(2d-i,i)$ K\"{u}nneth component of the 
diagonal class. 
\medskip

We say that a CK-decomposition as above satisfies {\em Poincar\'{e} duality} if $\pi_{2d-i}={\pi_i}^t$ for $0 \leq i \leq 2d$.

Finally, we recall the conjectures of Murre, formulated in \cite{mur} for smooth varieties:

\medskip

{\bf Murre's Conjectures}

\smallskip

Let $X$ denote a pseudo-smooth projective variety.  Then 

\textbf{A}. $X$ has a CK-decomposition $\ds [\Delta_X]= \sum_{i=0}^{2d} \pi_i$.  

\textbf{B}. If $i<j$ or $i>2j$, then $\pi_i$ acts as $0$ on $\rch^j(X)$.

\textbf{B'}. If $i<j$ or $i>j+ \dim X$, then $\pi_i$ acts as $0$ on $\rch^j(X)$.  

\textbf{C}.  If we define $F^0 \rch^j(X) = \rch^j(X)$ and $\ds F^k 
\rch^j(X) = \ker {\pi_{2j+1-k}}_*|_{F^{k-1} \rch^j(X)}$ for $k>0$, then the resulting 
filtration is independent of the particular choice of projectors $\pi_i$.  

\smallskip

\textbf{D}.  For any filtration as defined in \textbf{C}, $F^1 \rch^j(X)$ is the 
subgroup of cycles in $\rch^j(X)$ homologically equivalent to zero.

\medskip

\subsection{Intersection theory on pseudo-smooth varieties }

One can define pullback maps,  
pushforward maps, and intersection products in the context of pseudo-smooth
varieties, and many basic results (including, in particular, the
projection formula) carry over from the smooth case into this setting, provided one uses rational coefficients; see  \cite{dbn} for details.  
For this reason, we use rational coefficients throughout this section, even though many of the results (appropriately rephrased) hold with integral coefficients in the smooth case.
In the interest of making our proofs more concise, we will work with correspondences as much as possible; however, it will occasionally serve intuition better to argue directly using pullback and pushforward maps.  To this end, we record the following ``dictionary" (cf. \cite[Proposition 16.1.1]{fu} and \cite[Example 1.7.6]{fu}) which allows us to go back and forth between these two interpretations.  

\begin{lemma}\label{dict}
Let $X,Y, Z$ be pseudo-smooth projective varieties over a field.
Suppose $f: X \arr Y$ and $g: Y \arr Z$ are morphisms, and $\alpha \in
\rch^*(X \times Y)$,  $\beta \in \rch^*(Y \times Z)$, $\gamma \in \rch^*(X \times Z)$.  Then the following formulas hold:

$$(1 \times g)_*(\alpha)= [\Gamma_g] \bullet \alpha, ~~  (f \times 1)^*(\beta)= \beta \bullet [\Gamma_f], ~~ (f \times 1)_*(\gamma) = \gamma \bullet [\Gamma_f^t], ~~  (1 \times g)^*(\gamma)= [\Gamma_g^t] \bullet \gamma.$$

\end{lemma}

An important observation is that composition of correspondences is well-behaved with respect to pullback of cycles. 

\medskip

\begin{lemma}\label{compat}

With notation as in Lemma \ref{dict}, suppose further that $g$ is a morphism of degree $d$, and $\alpha, \beta \in \rch^*(Z \times Z)$.  Then $\ds (g \times g)^* \alpha \bullet (g \times g)^* \beta = d (g \times g)^*(\alpha \bullet \beta)$.

\end{lemma}

\proof

Observe first that by the projection formula, we have: 
$$[\Gamma_g] \bullet [\Gamma_g^t] = [\Delta_Z] \bullet [\Gamma_g]
\bullet [\Gamma_g^t] = (g \times 1)_* (g \times 1)^* [\Delta_Z] = d [\Delta_Z].$$
Then 

\begin{align} 
\begin{split}
(g \times g)^* \alpha \bullet (g \times g)^* \beta  &= (1 \times g)^* (g \times 1)^* \alpha \bullet (1 \times g)^* (g \times 1)^* \beta  \\ &=[\Gamma_g^t] \bullet \alpha \bullet [\Gamma_g] \bullet [\Gamma_g^t] \bullet \beta \bullet [\Gamma_g] \\ 
&=d \left( [\Gamma_g^t] \bullet \alpha \bullet \beta \bullet [\Gamma_g] \right)=d (g \times g)^* (\alpha \bullet \beta). \end{split}
\end{align}

\qed

\medskip

The following fact about compositions of product cycles is surely well known; however, since we will be using it so frequently, we include a proof in the interest of completeness of exposition.  

\begin{lemma}\label{orthoprod}

Let $X$ be a pseudo-smooth irreducible projective variety of dimension $d$ over some field $k$.  Suppose 
$\alpha \in \rch^i(X)$, $\beta \in \rch^j(X)$, $\gamma \in \rch^k(X)$, and $\delta 
\in \rch^{\ell}(X)$.  Then 

$$(\alpha \times \beta) \bullet (\gamma \times \delta)= (\gamma \times \beta) 
\cdot {p_{\emptyset}^{XX}}^*  {p_{\emptyset}^X}_* ( \delta \cdot \alpha).$$
 
In particular, $$(\alpha \times \beta) \bullet (\gamma \times \delta)=  m(\alpha, \delta) (\gamma \times \beta)$$ for some $m(\alpha, \delta) \in \ratls$, which equals zero if $i+\ell \neq d$.

\end{lemma}

\begin{proof}

\begin{align}
  \begin{split}
(\alpha \times \beta) \bullet (\gamma \times \delta) &={p_{13}^{XXX}}_*({p_{12}^{XXX}}^* ( {p_1^{XX}}^* \gamma \cdot {p_2^{XX}}^* \delta) \cdot {p_{23}^{XXX}}^* ( {p_1^{XX}}^* \alpha \cdot {p_2^{XX}}^*
\beta))\\
&={p_{13}^{XXX}}_*({p_{13}^{XXX}}^* {p_1^{XX}}^* \gamma \cdot {p_{2}^{XXX}}^* 
\delta \cdot {p_2^{XXX}}^* \alpha  \cdot {p_{13}^{XXX}}^* {p_2^{XX}}^* \beta) \\
&= {p_{13}^{XXX}}_*({p_{13}^{XXX}}^* ({p_1^{XX}}^* \gamma  \cdot 
{p_2^{XX}}^* \beta) \cdot {p_{2}^{XXX}}^* 
\delta \cdot {p_2^{XXX}}^* \alpha)\\
&={p_1^{XX}}^* \gamma  \cdot {p_2^{XX}}^* \beta \cdot 
{p_{13}^{XXX}}_* {p_{2}^{XXX}}^*(\delta  \cdot \alpha)\\
&={p_1^{XX}}^* \gamma  \cdot  {p_2^{XX}}^* \beta \cdot 
{p_{\emptyset}^{XX}}^*  {p_{\emptyset}^X}_* ( \delta \cdot 
\alpha)\\
&=(\gamma \times \beta) \cdot {p_{\emptyset}^{XX}}^*  {p_{\emptyset}^X}_* 
(\delta \cdot \alpha). \notag \end{split} \end{align}

Now let $m(\alpha, \delta)= {p_{\emptyset}^{XX}}^*  
{p_{\emptyset}^X}_* (\delta \cdot \alpha)$.  
If $i +\ell \neq d$, then ${p_{\emptyset}^X}_* ( \delta \cdot \alpha) \in 
\rch^{i+\ell-d}(\spec k)=0$.  If $i+\ell=d$, then ${p_{\emptyset}^{XX}}^*  
{p_{\emptyset}^X}_* (\delta \cdot \alpha) \in \rch^0(X \times X) \cong \ratls$.

\qed \end{proof}

\subsection{Lefschetz algebra structure}\label{lefalgebra}

We recall, with slight revisions, the definition of Lefschetz algebra from \cite{ku}.

\begin{definition}\label{lefalg}
A {\em Lefschetz algebra} of dimension $d$ is a triple $(R, \{\eta_i\}_{i=0}^{2d},  L, \Lambda)$ where $R=\directsum_{p \in \zah} R^p$ is a graded $\ratls$-algebra, $L \in R^1$, $\Lambda \in R^{-1}$, and

\begin{enumerate}

\item[(i)]

$\eta_0, \ldots, \eta_{2d}$ are elements of $R^0$ satisfying  $$\ds \sum_{i=0}^{2d} \eta_i = 1 \mbox{ and } \eta_i \circ \eta_j = \left\{ \begin{array}{cc}    \eta_i & \mbox{
      if } i=j \\ 0 & \mbox{ if } i \neq j. \\ \end{array} \right.$$

\item[(ii)]

For all $i$, $L \circ \eta_i = \eta_{i+2} \circ L$.

\item[(iii)]

For all $i$, $\Lambda \circ \eta_i = \eta_{i-2} \circ \Lambda$.

\item[(iv)]

$[\Lambda, L]: = \Lambda \circ L - L \circ \Lambda = \sum_{i=0}^{2d} (d-i) \eta_i$.

\end{enumerate} 

\end{definition}

The examples of primary concern to us arise when $X$ is a pseudo-smooth projective variety of dimension $d$ over some field, $R$ is the ring $\rch^{*+d}(X \times X)$, the $\eta_i$ are Chow-\kuns components of $[\Delta_X]$, and $L \in \rch^{d+1}(X \times X)$, $\Lambda \in \rch^{d-1}(X \times X)$ are elements satisfying the identities (ii)-(iv).  If $\rch^{*+d}(X \times X)$ can be endowed with the structure of a Lefschetz algebra in this manner, we say that $X$ is of {\em Lefschetz type}.  Varieties of Lefschetz type are of interest largely due to the following result, which may be deduced formally from the definition.

\medskip

\begin{corollary}\cite[Theorem 4.1]{ku}
Let $R$ be a Lefschetz algebra as above, and define \\ $I=\{(i,k) \in \zah \times \zah ~ | ~ \max \{0, i-d\} \leq k \leq \lfloor i/2 \rfloor \}$.
Then $R$ has a Lefschetz decomposition, i.e. there exist elements $p_{i,k} \in R^0$ satisfiying:

\begin{enumerate}

\item[(i)]$ \sum _{k} q_{i, k} = \eta_i$ for each $i$.
\item[(ii)] $q_{i, k} \circ \eta_j = \eta_j \circ q_{i, k} = q_{i, k}$ if $i=j$ 
and $0$
otherwise.
\item[(iii)] $q_{i, k} = 0$ for $(i, k) \nin  I$.
\item[(iv)] $q_{i, k} \circ q_{j, l} = q_{i, k}$ if $i=j$ and $k =l$ and $0$
otherwise.
\item[(v)] $q_{i, k} \circ L = L \circ q_{i-2,
k-1}$.
\item[(vi)] $\Lambda \circ q_{i, k} = q_{i-2, k-1} \circ
\Lambda$.
\item[(vii)] $L \circ \Lambda \circ q_{i, k} = k(g-i +k+1) q_{i,
k}$.
\item[(viii)] $\Lambda \circ L \circ q_{i, k} = (k+1) (g-i+k) q_
{i, k}$.
\end{enumerate}

\end{corollary}

\begin{corollary}\cite[Theorem 5.2]{ku}(Hard Lefschetz Theorem)
If $X$ is a pseudo-smooth projective variety of dimension $d$ over a field $k$, and  $R=((\rch^*(X \times X), \{\eta_i\}_{i=0}^{2d}, L, \Lambda)$ is a Lefschetz algebra, then for $i$, $0 \leq i \leq d$, the correspondence $L^{d-i}$ defines an isomorphism of motives $h^i(X) \stackrel{\cong}{\larr} h^{2d-i}(X)(d-i)$ with inverse $\Lambda^{d-i}$, where $h^j(X)$ is the rational Chow motive $(X, \eta_j,0)$.  

\end{corollary}

\section{\bf An explicit Chow-\kuns decompositions for blow-ups}\label{ck}

\subsection{Main construction and technical details}\label{maincon}

For the balance of the paper, we fix the following notation and hypotheses.

\medskip

\textbf{Assumptions.}

\begin{itemize}
\label{key.assumpns}
\item

$X$ is a pseudo-smooth projective variety of dimension $d$ over some field $k$.

\item

$Y$ is the blow-up of $X$ along $T=\{a\}$, where $a$ is a $k$-rational point of $X$.

\item

$Y$ is pseudo-smooth and the exceptional divisor $Z$ of the blow-up is isomorphic to $\psd$.

\end{itemize}

Let  

\xymatrix{ & E \ar[r]^j \ar[d]^g & Y \ar[d]^f  \\ & T \ar[r]^i & X } be the commutative square describing this blow-up.  

\medskip

The objective of this section is to describe explicitly how a CK-decomposition for $X$ can be used to construct one on $Y$, and conversely.  We begin by setting up the framework for our construction and proving some auxiliary results.  

\vskip .3cm

\medskip

First, observe that $Y \times Y$ is the blow-up of $X \times X$ along the closed subscheme $S=S_1 \cup S_2$, where $S_1=T \times X$ and $S_2 = X 
\times T$.  The exceptional divisor of this blow-up is  $E=E_1 \cup E_2$, where 
$E_1=Z \times Y$ and $E_2=Y \times Z$. Thus we have  commutative diagrams:

\xymatrix{ E \ar[rr]^{\tilde{j}} \ar[dd]^{\tilde{g}} & & Y \times Y
\ar[dd]^{f \times f} & & & E_1 \ar@{^(->}[rr]^{j \times 1}  \ar[dd]^{g \times 1} & &  Y \times 
Y \ar[dd]^{f \times 1} & & 
 E_2 \ar@{^(->}[rr]^{1 \times j}  \ar[dd]^{1 \times g} & &  Y \times 
Y \ar[dd]^{1 \times f} 
\\ 	
\\ S \ar[rr]^{\tilde{i}} & & X \times X  & & & T \times Y \ar@{^(->}[rr]^{i \times 1} & & X \times Y & &  Y \times T \ar@{^(->}[rr]^{1 
\times i} & & Y \times X. }

\bigskip

Note that even when $X$ is smooth, $S$ is not regularly imbedded in $X \times X$, so we cannot use the blow-up exact sequence 
to 
relate the Chow groups of $Y \times Y$ to those of $X \times X$.  Instead, we
use the localization sequence.  Let $m$ be an integer, $0 \leq m \leq d$, and define $U_X=X \times X - S$ and $U_Y= Y \times Y -E$.  Then by 
there is a commutative diagram with exact rows:
\begin{equation}
 \label{local.diag}
\xymatrix{ 
& \rch^{m-1}(E) \ar[rr]^{{\tilde{j}}_*} \ar[dd]^{{\tilde{g}}_*} & & \rch^m(Y \times 
Y) \ar[rr] \ar[dd]^{(f \times f)_*} & & \rch^m(U_Y) \ar[dd]^{\cong} \ar[r] & 0  \\ \\
& \rch^{m-d}(S) \ar[rr]^{{\tilde{i}}_*} & & \rch^m(X \times X) \ar[rr] & & \rch^m(U_X) 
\ar[r] & 0  
} 
\end{equation}

\medskip

The aim of the rest of this section is to describe a decomposition of $\rch^*(Y \times Y)$ as the internal direct sum of two subgroups, $A$ and $B$, and to study the multiplicative structure of $\rch^*(Y \times Y)$ as a ring (under composition of correspondences) with respect to these subgroups.

\medskip

To this end, define 
\begin{equation}
 \label{xi.def}
\zeta=(f \times f)^*[\Delta_X] \in \rch^d(Y \times Y), ~ ~   A= \zeta \bullet \rch^*(Y \times Y) \bullet \zeta, \mbox{ and }  A_m = A \cap \rch^m(Y \times Y), ~ 0 \leq m \leq 2d.
\end{equation}
Clearly, $A=\directsum_{m=0}^{2d} A_m$.

\begin{lemma}\label{abouta}
For $\gamma \in \rch^*(Y \times Y)$, $(f \times f)^* (f \times f)_* \gamma = \zeta \bullet \gamma \bullet \zeta$.  In particular, if $\gamma \in (f \times f)^* \rch^*(X \times X)$, then $\zeta \bullet \gamma \bullet \zeta = \gamma$.  Furthermore, $A=(f \times f)^* \rch^*(X \times X)$ is a ring with unit element $\zeta$, and $A_d$ is a subring of $A$.
\end{lemma}

\proof

From the projection formula and the fact that $f$ has degree $1$, it follows that $[\Gamma_f] \bullet [\Gamma_f^t] = [\Delta_X]$.  Then 
$$\zeta \bullet \gamma \bullet \zeta = (f \times f)^* [\Delta_X] \bullet \gamma \bullet (f \times f)^* [\Delta_X]
= [\Gamma_f^t] \bullet [\Delta_X] \bullet [\Gamma_f] \bullet \gamma \bullet [\Gamma_f^t] \bullet [\Delta_X] \bullet [\Gamma_f] 
= [\Gamma_f^t] \bullet [\Gamma_f] \bullet \gamma \bullet [\Gamma_f^t]  \bullet [\Gamma_f] .$$
$$=(f \times f)^*(f \times f)_* \gamma.$$

The remaining assertions are clear from Lemma \ref{compat}.

\qed

\endproof

Next, define $$B'=(j \times 1)_* (\ker ((f \circ j) \times 1)_*) $$ $$B''=(1 \times j)_*  (\ker (1 \times (f \circ j))_*)$$

Let $B=B'+B''$; for $0 \leq m \leq 2d$, set $B'_m =B' \cap \rch^m(Y \times Y)$,  $B''_m=B''  \cap \rch^m(Y \times Y)$, and $B_m = B'_m  + B''_m $.  Then there are direct sum decompositions $$B=\directsum_{m=0}^{2d} B_m , ~ B'=\directsum_{m=0}^{2d} B'_m, ~  B''=\directsum_{m=0}^{2d} B''_m.$$

\medskip

There is a rather important orthogonality relationship between $A$ and $B$.

\begin{proposition}\label{weakor}(Orthogonality principle)
Suppose $\alpha \in A$, $\beta' \in B'$ and $\beta'' \in B''$.  Then $\beta' \bullet \alpha=0$ and $\alpha \bullet \beta''=0$.  
\end{proposition}

\proof

Write $\alpha=(f \times f)^* \delta$, $\beta'=(j \times 1)_*
\epsilon_1$ and $\beta''=(1 \times j)_* \epsilon_2$, where $\delta \in
\rch^d(X \times X)$, $\epsilon_1 \in \ker ((f \circ j) \times 1)_* $, and
$\epsilon_2 \in \ker  (1 \times (f \circ j))_*  $.  Then 

$$\beta' \bullet \alpha =  \epsilon_1 \bullet [\Gamma_j^t] \bullet   [\Gamma_f^t] \bullet \delta \bullet [\Gamma_f] = ((f \circ j) \times 1)_* \epsilon_1 \bullet \delta \bullet [\Gamma_f] = 0.$$

$$\alpha \bullet \beta'' =[\Gamma_f^t] \bullet \delta \bullet [\Gamma_f] \bullet [\Gamma_j] \bullet \epsilon_2 =[\Gamma_f^t] \bullet \delta \bullet (1 \times (f \circ j))_* \epsilon_2 = 0.$$ \qed

A simple chase on diagram \ref{local.diag} shows that $\sigma=[\Delta_Y] - \zeta \in B_d$.  Direct calculation then shows that the formulas $$\sigma \bullet \sigma=\sigma, ~ \sigma^t=\sigma, ~ \mbox{ and } \sigma \bullet \zeta = \zeta \bullet \sigma = 0$$

hold.  Moreover, Proposition \ref{weakor} shows that for $\beta' \in B'$, $\beta'' \in B''$, we have $\beta' \bullet \sigma=\beta'  \mbox{ and } \sigma \bullet \beta''=\beta''.$

\medskip

Given $\gamma \in  \rch^*(Y \times Y)$, let $\gamma' = \gamma - \zeta \bullet \gamma \bullet \zeta$.  Then using Lemma \ref{abouta} we calculate: $$(f \times f)_*(\gamma') = (f \times f)_* (\gamma - \zeta \bullet \gamma \bullet \zeta ) = (f \times f)_*  \gamma - (f \times f)_* (f \times f)^* (f \times f)_* \gamma = (f \times f)_* \gamma - (f \times f)_* \gamma = 0.$$
Another chase on diagram \ref{local.diag} shows that $\gamma' \in \im {\tilde{j}}_* \cap \ker (f \times f)_*$.  This shows that there are well-defined maps $$s: A\oplus B \arr \rch^*(Y \times Y) \mbox{ and } t: \rch^*(Y \times Y) \arr A \oplus B$$ given by  $s(a,b)=a+b$ and $t(\gamma)=(\zeta \bullet \gamma \bullet \zeta, \gamma - \zeta \bullet \gamma \bullet \zeta)$.

\begin{proposition}\label{dstrans}
~
\begin{enumerate}

\item

$s$ and $t$ are inverse isomorphisms; thus, $\rch^*(Y \times Y)$ is the internal direct sum of $A$ and $B$.  

\item
$A$ and $B$ are each closed under transposition of cycles.

\end{enumerate}

\end{proposition}

\proof

It is clear from the definitions that $s \circ t = 1$, so it suffices to show that $s$ is injective and that $t$ is surjective.  To show the former, we prove $A \cap B=\{0\}$.  If $\gamma \in A \cap B$, then in particular, $\gamma \bullet \zeta = \gamma = \zeta \bullet \gamma$, and also $\gamma=b' + b''$ for some $b' \in B'$ and $b'' \in B''$.  Using Proposition \ref{weakor}, we calculate:

$$\gamma=\zeta \bullet \gamma \bullet \zeta = \zeta \bullet (b' \bullet \zeta) + (\zeta \bullet b'') \bullet \zeta = 0.$$  

Now suppose $(\alpha,\beta) \in A \oplus B$.  Setting $\gamma=\alpha+\beta$ and writing $\beta = \beta' + \beta''$ with $\beta' \in B'$ and $\beta'' \in B''$, we have $$\zeta \bullet \gamma \bullet \zeta = \zeta \bullet \alpha \bullet \zeta + \zeta \bullet \beta \bullet \zeta = \alpha + \zeta \bullet (\beta' \bullet \zeta) + (\zeta \bullet \beta'') \bullet \zeta = \alpha.$$

Hence $t(\gamma)=(\alpha, \beta)$, and so $t$ is surjective.

\medskip

For the second statement, $\alpha \in A$ implies $\alpha
= (f \times f)^* \delta$ for some $\delta \in \rch^*(X \times X)$.
Since $\delta^t \in \rch^*(X \times X)$, obviously
$\alpha^t = (f \times f)^* \delta^t \in A$.  Furthermore, suppose $\beta \in B$ and write $\beta^t = \alpha' + \beta'$ for
some $\alpha' \in A$, $\beta' \in B$.  Then $\beta=(\beta^t)^t =
{\alpha'}^t + {\beta'}^t$.  Because $\beta$ has a unique expression as a sum of an element of $A$ and an element of $B$  we must have ${\alpha'}^t=0$; so $\alpha'=0$, and thus $\beta^t \in B$. \qed

\endproof

The following is a computational criterion convenient for testing for membership in $A$ or $B$:   

\begin{corollary}\label{abtest}
Suppose $\gamma \in \rch^*(Y \times Y)$.  Then $\gamma \in A$ if and only if $\zeta \bullet \gamma \bullet \zeta=\gamma$ and $\gamma \in B$ if and only if $\zeta \bullet \gamma \bullet \zeta = 0$.  
\end{corollary}

\proof

If $\gamma \in A$, then $\gamma = \zeta \bullet \gamma \bullet \zeta$ by Lemma \ref{abouta}.  Conversely, suppose $\gamma=\zeta \bullet \gamma \bullet \zeta$.  Write $\gamma = \alpha + \beta$, where $\alpha \in A$ and $\beta \in B$ and $\beta=\beta' + \beta''$, where $\beta' \in B'$ and $\beta'' \in B$.  Then, using Proposition \ref{weakor}, $$\gamma = \zeta \bullet \gamma \bullet \zeta = \zeta \bullet \alpha \bullet \zeta + \zeta \bullet \beta' \bullet \zeta + \zeta \bullet \beta'' \bullet \zeta = \zeta \bullet \alpha \bullet \zeta = \alpha \in A.$$  The second statement is now clear.

\qed

\endproof

\subsection{Blowing up}

The proof of the following theorem introduces one of the main constructions used in this article.  Our techniques bear 
some resemblance to those used by Vial in \cite[Section 5]{v}.  There are, however,  two main  differences in approach.
\vskip .2cm
(i) We are blowing up varieties that are not necessarily smooth, but having isolated quotient singularities, so that our varieties are 
required to be only pseudo-smooth. This also means that we cannot make use of the blow-up exact sequence (as in, for e.g. \cite[Proposition 6.7]{fu}).
\vskip .2cm
(ii) The second difference 
 is that 
we interpret  $Y \times Y$ as the blow-up of $X \times X$ along the subscheme $S$ (as explained in Section \ref{maincon}); in contrast, the 
arguments of Vial only involves $Y$ as the blow-up of $X$ along a subvariety $T$.

\begin{theorem}\label{blowup}
If $X$ has a CK-decomposition $\ds [\Delta_X]=\sum_{i=0}^{2d} \pi_i^X$, then $Y$ has an explicit  CK-decomposition, as defined in ~\eqref{CK.for.Y}.  If the former satisfies Poincar\'{e} duality, then so does the latter.   
\end{theorem}

\begin{proof}

By Lemma \ref{compat}, $\{(f \times f)^* \pi_i^X \}_{i=0}^{2d}$ is a set of 
orthogonal idempotents in $\rch^d(Y \times Y)$. However, $\Sigma_{i=0}^{2d} (f \times f)^*(\pi_i^X) $ may not equal  $[\Delta_Y]$, so we proceed to deal with this discrepancy.  Recall $\sigma =[\Delta_Y] - \zeta =[\Delta_Y]-(f \times f)^*([\Delta _X])$.
Denote by $h_1: E_1 \incl E$, $h_2: E_2 \incl E$, $k_1: E_1 \cap E_2 \incl E_1$, and $k_2: E_1 \cap  E_2 \incl E_2$ the various inclusion maps.  
Then the sequence  
\begin{eqnarray}\label{intersection}
\ds \rch^{d-2}(E_1 \cap E_2) \stackrel{{k_1}_* + {k_2}_*}{\arr} \rch^{d-1}(E_1) \oplus \rch^{d-1}(E_2) \stackrel{{h_1}_* - {h_2}_*}{\larr} \rch^{d-1}(E) \arr 0
\end{eqnarray}
is exact, so we may write $\sigma=\tilde{j}_* ({h_1}_* \tau_1 + {h_2}_* \tau_2)$, with $\tau_i \in 
\rch^{d-1}(E_i)$, $i=1, 2$.  Since $\tilde{j} \circ h_1 = j \times 1$ and 
$\tilde{j} \circ h_2 = 1 \times j$, we have $\sigma = (j \times 1)_* \tau_1 + (1 
\times j)_* \tau_2$.  We stress that this is the only part of our construction which involves a choice of cycles.  

\medskip \medskip

By Proposition \ref{dstrans}, we have also $\sigma=\sigma^t = (1 \times j)_* \tau_1^t + (j \times 1)_* \tau_2^t$; thus, 
$$\sigma = \frac{1}{2}[ (j \times 1)_* \tau_1 + (1 \times j)_* \tau_2] + \frac{1}{2} [ (1 \times j)_* \tau_1^t + (j \times 1)_* \tau_2^t)] = (j \times 1)_* \frac{1}{2}(\tau_1 +\tau_2^t) + (1 \times j)_* \frac{1}{2}(\tau_1 + \tau_2^t)^t.$$

This calculation shows that we may replace $(\tau_1, \tau_2)$ with $\ds (\frac{1}{2}(\tau_1 + \tau_2^t), \frac{1}{2}(\tau_1^t + \tau_2))$, and thus assume without loss of generality that $\tau_2=\tau_1^t$.

\medskip

Let $\ell \in \rch^1(Z)$ be the class of a generic
hyperplane, and for convenience, set $\ell_i=j_*(\ell^{i-1}) \in \rch^i(Y)$ for $1
\leq i \leq d$.  From the projective bundle formula, we have $\tau_1 =
\sum_{i=0}^{d-1} \ell^i \times a_{d-i-1}$, where $a_{d-i-1} \in
\rch^{d-i-1} (Y)$.  Define $a_d=0$, $\ell_0=0$, $\eta_0=0$ and
$\eta_i=(j \times 1)_* (\ell^{i-1} \times a_{d-i}) = \ell_i \times
a_{d-i}$ for $1 \leq i \leq d$.  If we set $\theta_i =
\eta_{d-i}^t$, then $\eta_i$ is a product cycle of type $(i, d-i)$
when $1 \leq i \leq d$ and $\theta_i$ is a product cycle of type
$(i,d-i)$ when $0 \leq i \leq d-1$.      
Finally, define $\gamma_i=\eta_i + \theta_i$ for $0 \leq i \leq d$.
By construction, we have $\gamma_i^t = \gamma_{d-i}$.

\begin{lemma}
$a_0=0$.  
\end{lemma}

\proof

Since $a_0 \in \rch^0(Y)$, we have $a_0=c[Y]$ for some $c \in \ratls$.  
Observe first that $$(f \times f)_*(j \times 1)_* \tau_1 = (1 \times
f)_* \sum_{m=0}^{d-1} (f \times 1)_* (j \times 1)_*  (\ell^m \times
a_{d-m-1}) = (1 \times f)_* \sum_{m=0}^{d-1} (i \times 1)_* (g \times
1)_* (\ell^m \times a_{d-m-1})$$ $$= (1 \times f)_* \sum_{m=0}^{d-1} i_* g_*(\ell^m) \times a_{d-m-1}.$$
For reasons of dimension, $g_*(\ell^m)=0$ when $0 \leq m \leq d-2$ and
$i_* g_* \ell^{d-1}=x \in \rch^d(X)$, so \\$(f \times f)_* (j \times
1)_* \tau_1 = x \times a_0 = c (x \times [Y])$ and similarly $(f
\times f)_* (1 \times j)_* \tau_2 = c([Y] \times x)$.   

\medskip

By construction, $(f \times f)_* \sigma=0$; thus, $(f \times f)_* (j
\times 1)_* \tau_1 + (f \times f)_* (1 \times j)_* \tau_2 = c( x
\times [Y] +  [Y] \times x) = 0$.  Since $x \times [Y]$ is a
projector which is orthogonal to $[Y] \times x$, it follows that $(x
\times [Y]) \bullet (c ( x \times [Y] +  [Y] \times x)) = c(x
\times [Y])=0$.  Finally, since $\Delta_Y^*(x \times Y) = x \cdot [Y] =
x \neq 0$, we must have $c=0$, and hence $a_0=0$. \qed 

\medskip

Returning to the proof of Theorem \ref{blowup}, we note the following important facts:
\begin{eqnarray}\label{inb}
\theta_0=\eta_d=0 \mbox{ and } \eta_i \in B_d', ~ \theta_i \in B_d'' \mbox{ for } 0 \leq i \leq d.
\end{eqnarray}

\medskip

One easily checks that $\sum_{j=0}^{d} \gamma_j = \sigma$.  Since $\gamma_i$ is a product cycle of type $(i, d-i)$, Lemma \ref{orthoprod} implies $\gamma_i \bullet \gamma_j=0$ when $i \neq j$.  
In particular, we have:
\begin{eqnarray}\label{gammass1}
\sigma \bullet \gamma_i = \gamma_i \bullet \gamma_i = \gamma_i \bullet \sigma 
\end{eqnarray}

for $i$, $0 \leq i \leq d$.  Thus, $\sigma = \sigma \bullet \sigma = \sum_{j=0}^d \gamma_j \bullet \gamma_j$, and so we have $\sum_{j=0}^d [\gamma_j \bullet \gamma_j  - \gamma_j]=0$, where the term in brackets is a product cycle of type $(j, d-j)$.  Composing with $\gamma_i$ on the left, we  conclude:
\begin{eqnarray}\label{twothree}
\gamma_i  \bullet \gamma_i \bullet \gamma_i - \gamma_i \bullet \gamma_i = 0.
\end{eqnarray}

Now by (\ref{gammass}), $\sigma \bullet (\gamma_i  \bullet \gamma_i - \gamma_i) = (\sigma \bullet \gamma_i) \bullet \gamma_i - \sigma \bullet \gamma_i = \gamma_i \bullet \gamma_i \bullet \gamma_i - \gamma_i \bullet \gamma_i = 0$, and similarly $(\gamma_i  \bullet \gamma_i - \gamma_i) \bullet \sigma = 0$.  Thus, 
$$\gamma_i \bullet \gamma_i  - \gamma_i = (\zeta + \sigma) \bullet (\gamma_i \bullet \gamma_i - \gamma_i) \bullet (\zeta + \sigma) = \zeta \bullet (\gamma_i \bullet \gamma_i - \gamma_i) \bullet \zeta,$$
which by (\ref{gammass1}), equals
$$\zeta \bullet \sigma \bullet \gamma_i \bullet \zeta + \zeta \bullet (\eta_i + \theta_i) \bullet \zeta.$$

The first term vanishes because $\zeta \bullet \sigma =0$ and the second vanishes because $\eta_i \in B'$ and $\theta_i \in B''$.  Thus, we have $\gamma_i \bullet \gamma_i = \gamma_i$, and so $\gamma_i$ is a projector satisfying
\begin{eqnarray}\label{gammass}
\sigma \bullet \gamma_i = \gamma_i = \gamma_i \bullet \sigma. 
\end{eqnarray}
Furthermore, for $i$ and $j$, $0 \leq i \leq d$ and $0 \leq j \leq 2d$, we have:
$$\gamma_i \bullet (f \times f)^* \pi^X_j = (\gamma_i \bullet \sigma) \bullet (\zeta \bullet (f \times f)^* \pi^X_j \bullet \zeta) = \gamma_i \bullet (\sigma \bullet \zeta) \bullet (f \times f)^* \pi^X_j \bullet \zeta= 0$$

and similarly $(f \times f)^* \pi^X_j \bullet \gamma_i=0$.

\medskip

Finally, define for $0 \leq j \leq 2d$, 
\begin{equation}
 \label{CK.for.Y}
\delta_j = \left\{\begin{array}{cc} \gamma_{d-\frac{j}{2}} & \mbox{ if } j \mbox{ is even} \\ 0 & \mbox{ if } j \mbox{ is odd} \end{array} \right. ~ ~  \mbox{ and }  ~ ~ ~ \pi^Y_j = (f \times f)^* \pi^X_j + \delta_j.
\end{equation}
\bigskip

The computations above show that $[\Delta_Y]= \sum_{j=0}^{2d}(f \times f)^*(\pi_j^X)+ \sigma = \sum_{j=0}^{2d}  \pi_j^Y$
satisfies properties (i) and (ii) in the definition of CK-decomposition.  The construction shows that $\delta_j^t = \delta_{2d-j}$ for all $j$; hence, the assertion about Poincar\'{e} duality follows from the definition of the $\pi_j^Y$.

\medskip

It remains to show that for any Weil cohomology theory 
$H^*$ and every $j$, $0 \leq j \leq 2d$, $cl_{Y \times Y}(\pi^Y_j)$ is the 
$(2d-j, j)$ \kuns component of $[\Delta_Y] \in H^{2d}(Y \times Y; \ratls)$.  
Using the \kuns isomorphism to make the identification $H^{2d}(Y \times Y; 
\ratls) \cong  \directsum_{i=0}^{2d} H^{2d-i}(Y; \ratls) \otimes_{\ratls} 
H^i(Y; \ratls)$, it suffices to show that $cl_{Y \times Y}(\pi_j^Y)  \in 
H^{2d-j}(Y; \ratls) \otimes_{\ratls} H^j(Y; \ratls)$.  

\medskip

Now $\pi_j^X$ is a projector in the original Chow-\kuns composition for $X$; 
so  $cl_{X \times X} (\pi_j^X) \in H^{2d-j}(X; \ratls) \otimes_{\ratls} H^j(X; \ratls)$.
Hence, using properties of the cycle map from the definition of Weil cohomology 
(see for example, \cite[Section 3]{kl}),  we have $cl_{Y \times Y} (f \times f)^*  
\pi^X_j  = (f \times f)^* cl_{X \times X} (\pi_j^X) \in H^{2d-j}(Y;\ratls) 
\otimes_{\ratls} H^j(Y; \ratls)$.  Moreover, $\delta_j=\gamma_{d-j/2}$ is a product 
cycle of type $(d-j/2, j/2)$; hence $\gamma_{d-j/2} = \sum_{m=0}^r  
\lambda_m \times \mu_m$, where $\lambda_m \in \rch^{d-j/2}(Y)$ and $\mu_m \in 
\rch^{j/2}(Y)$.  Again using properties of the cycle map,  

$$cl_{Y \times Y}  (\gamma_{d-j/2}) = \sum_{m=0}^r cl_{Y \times Y} (\lambda_m 
\times \mu_m) = \sum_{m=0}^r cl_Y(\lambda_m) \otimes cl_Y(\mu_m) \in 
H^{2d-j}(Y;\ratls) \otimes_{\ratls} H^j(Y; \ratls).$$  

Thus, regardless of whether $j$ is odd or even, $cl_{Y \times Y}(\pi^Y_j) \in 
H^{2d-j}(Y; \ratls) \otimes_{\ratls} H^j(Y; \ratls)$.  

\qed

\end{proof}

\subsection{Refined projectors}

In the construction of the Chow-\kuns projectors, our interest was focused on the sums $\gamma_i = \eta_i + \theta_i$ as projectors in $\rch^d(Y \times Y)$.  In subsequent sections, however, we will need to use the fact, established below, that $\eta_i$ and $\theta_i$ are themselves mutually orthogonal projectors.  Before proceeding any further, we need to make some slight modifications to our definitions in certain cases, the reasoning being that we wish to avoid the situation in which $\eta_i$ and $\theta_i$ are nonzero constant multiples of each other.
To this end, if $0< i < d/2$ and $\eta_i$ and $\theta_i$ happen to be nonzero constant multiples
of each other, replace $(\eta_i, \theta_i)$ with $(\eta_i + \theta_i,
0)$ and replace $(\eta_{d-i}, \theta_{d-i})$ with $(0, \eta_{d-i} +
\theta_{d-i})$.  This change alters the definition of the $a_i$, but it does not change $\gamma_i$ or $\gamma_{d-i}$, nor does it disturb the duality relation $\eta_i^t=\theta_{d-i}$.   Now if $d$ is even, $\eta_{d/2}^t= \theta_{d/2}$.  In view of this, the only way for
$\eta_i$ and $\theta_i$ to be nonzero constant multiples of
each other (after the modification described above) is when $d$ is even, $i=d/2$ and $\eta_{d/2}=\theta_{d/2}$  We will show in Proposition \ref{mvalues} that this situation is impossible.  

\begin{lemma}\label{ellpowers}
For $i$, $1 \leq i \leq d-1$, $\ell_i \times \ell_{d-i} \neq 0$.  
\end{lemma}

\proof

If $\ell_i \times \ell_{d-i}=0$, then ${\Delta}_Y^*(\ell_i \times \ell_{d-i}) = \ell_i \cdot \ell_{d-i} =0$.  However, by the projection formula and the self-intersection formula, $$\ell_i \cdot \ell_{d-i} = j_* \ell^{i-1} \cdot j_* \ell^{d-i-1} = j_*(\ell^{i-1} \cdot j^* j_*\ell^{d-i-1}) = j_*(\ell^{i-1} \cdot \ell^{d-i})= j_*(\ell^{d-1}).$$ This is a zero cycle on $Y$ of degree one, so it cannot be zero.  \qed

\begin{proposition}\label{mvalues}

For all $i$, $0 \leq i \leq d$, $$\eta_i \bullet \eta_i = \eta_i, ~ \theta_i \bullet \theta_i = \theta_i, ~ \eta_i \bullet \theta_i = \theta_i \bullet  \eta_i = 0.$$  

Moreover, $m(\ell_i, a_{d-i}) =m(a_{d-i}, \ell_i) = m(\ell_i,
\ell_{d-i})=1$, $m(a_i, a_{d-i})=0$, where the $m(- , -)$ are the rational numbers defined in Proposition \ref{orthoprod}.  For all $i$ and $j$,  $a_i \times a_j=0$.

\end{proposition}

\proof

Direct calculation using the self-intersection formula shows that 
$$m(\ell_i, \ell_{d-i}) = m(j_* \ell^{i-1}, j_* \ell^{d-i-1}) = {p_{\emptyset}^{YY}}^* {p_{\emptyset}^Y}_* (j_* \ell^{i-1} \cdot j_* \ell^{d-i-1})  = {p_{\emptyset}^{YY}}^* {p_{\emptyset}^Y}_* j_* (j^* \ell^{i-1} \cdot \ell^{d-i-1})$$ $$= {p_{\emptyset}^{YY}}^* {p_{\emptyset}^Y}_* j_* (\ell^i \cdot \ell^{d-i-1})  =  {p_{\emptyset}^{YY}}^* {p_{\emptyset}^Y}_* j_* (\ell^{d-1})=1$$
for $i$, $0 \leq i \leq d$.  Moreover, since $\eta_i$ and $\theta_i$ are product cycles on $Y \times Y$,
Lemma \ref{orthoprod} implies that $\eta_i \bullet \eta_i = s \eta_i$ and $\theta_i \bullet \theta_i = t \theta_i$ for some $s,t \in \ratls$.  
Now $\eta_i \in B'$ and $\theta_i \in B''$ by (\ref{inb}), so by Proposition \ref{weakor}, we have 
$\eta_i \bullet (f \times f)^*\pi_j = 0$ and $(f \times f)^* \pi_j \bullet \theta_i = 0$ for $0 \leq j \leq 2d$.  Using Lemma \ref{orthoprod}, we have 
\begin{align}
   \label{key.rels.1}
\eta_i \bullet \gamma_i &= \eta_i \bullet \sigma = \eta_i \bullet [\Delta_Y] - \eta_i \bullet \sum_{j=0}^{2d} (f\times f)^* \pi_j = \eta_i, \mbox{ and similarly}\\
\gamma_i \bullet \theta_i &= \theta_i. \notag
\end{align}
\vskip .2cm
 So on one hand, $\eta_ i \bullet \theta_i = \eta_i
\bullet \gamma_i - \eta_i \bullet \eta_i = (1-s) \eta_i$, but also
$\eta_i \bullet \theta_i = \gamma_i \bullet \theta_i - \theta_i \bullet \theta_i = (1-t) \theta_i$.  Hence $(1-s) \eta_i = (1-t) \theta_i$.  

\medskip

First suppose $\eta_i$ and $\theta_i$ are not nonzero constant
multiples of each other.  It must be the case that $s=1$ or $t=1$.  If $s=1$, then $(1-t) \theta_i=0$, so $t=1$ or
$\theta_i=0$; but in the latter case, we may still assume $t=1$.  If
$t=1$, we may similarly conclude that $s=1$.  Thus, $s=t=1$, and so $\eta_i \bullet \eta_i = \eta_i$, $\theta_i \bullet \theta_i=\theta_i$ and $\eta_i \bullet \theta_i = 0$.  Now $\theta_i \bullet \eta_i = (\gamma_i - \eta_i) \bullet (\gamma_i - \theta_i) = \gamma_i \bullet \gamma_i - \eta_i \bullet \gamma_i - \gamma_i \bullet \theta_i + \eta_i \bullet \theta_i = \gamma_i - \eta_i - \theta_i = 0.$
Now the formulas $m(\ell_i, a_{d-i}) =m(a_{d-i}, \ell_i)=1$ follow
from the statement $\eta_i \bullet \eta_i = \eta_i$ and the symmetry
of $m(-, -)$ in its arguments.  Since $\eta_i \bullet
\theta_i=0$, we have $m(a_i, a_{d-i}) (\ell_i \times \ell_{d-i}) = 0$.
If $1 \leq i \leq d-1$, then by Lemma \ref{ellpowers}, $\ell_i \times
\ell_{d-i} \neq 0$, so we must have $m(a_i, a_{d-i})=0$.  When $i=0$ or $i=d$,
this is obvious, since $a_0=a_d=0$.  Finally, from $\theta_i \bullet
\eta_i=0$, we have $m(\ell_i, \ell_{d-i}) (a_i \times a_{d-i})=a_i
\times a_{d-i}=0$.  Then

$$a_i \times a_j = (\ell_i \times a_j) \bullet (a_i \times a_{d-i}) = 0.$$

Thus, all of asserted equations hold for $i \neq d/2$, or if $i=d/2$
and either $\eta_{d/2}=\theta_{d/2}=0$ or $\eta_{d/2} \neq \theta_{d/2}$.

\medskip

Now suppose that $d$ is even and $\eta_{d/2}=\theta_{d/2} \neq 0$.  Then $\eta_{d/2}=\theta_{d/2}=\frac{1}{2} \gamma_{d/2}$, and so $\eta_{d/2} \bullet \theta_{d/2} = \frac{1}{4} \gamma_{d/2} = \frac{1}{2} \eta_{d/2} \neq 0.$  However, by Lemma \ref{orthoprod},
$$\eta_{d/2} \bullet \theta_{d/2}  = m(\ell_{d/2}, \ell_{d/2}) a_{d/2}
\times a_{d/2} = a_{d/2} \times a_{d/2} = (\ell_{d/2 + 1} \times
a_{d/2}) \bullet (a_{d/2} \times a_{d/2-1})$$ $$=  (\ell_{d/2 + 1}
\times a_{d/2})  \bullet (a_{d/2 + 1} \times a_{d/2 -1}) \bullet (a_{d/2} \times \ell_{d/2-1}).$$
Since we have already showed that $a_{d/2 + 1} \times a_{d/2 -1} =0$, the above calculation 
forces $\gamma_{d/2}=0$, which is a contradiction.  This shows that
the condition $\eta_{d/2}=\theta_{d/2} \neq 0$ is impossible. 

 \qed
\endproof

\begin{corollary}\label{onesided}
$ $ 
\begin{enumerate}
\item 
Let $\beta_1= \ell_i \times \lambda \in \rch^{i+j}(Y \times Y)$.  Then $\beta_1 \bullet \eta_i = \beta_1$.  If $\beta_1 \in B'$ or if $a_i \times \lambda = 0$, then $\beta_1 \bullet \theta_i = 0$.   

\item
Let $\beta_2= \lambda \times \ell_j \in \rch^{i+j}(Y \times Y)$.  Then
$\theta_{d-j} \bullet \beta_2 =\beta_2$.  If $\beta_2 \in B''$ or if
$\lambda \times a_j = 0$, then $\eta_{d-j} \bullet \beta_2 = 0$.

\end{enumerate}
\end{corollary} 

\proof

We prove the first statement, the second being similar.  First, 
$$\beta_1 \bullet \eta_i = (\ell_i \times \lambda) \bullet (\ell_i \times a_{d-i})= m(\ell_i, a_{d-i}) (\ell_i \times \lambda) = \ell_i \times \lambda = \beta_1$$
By direct computation, $\beta_1 \bullet \theta_i = m(\ell_i,
\ell_{d-i})(a_i \times \lambda) =a_i \times \lambda$; so if $a_i
\times \lambda=0$, then $\beta_1 \bullet \theta_i=0$.  If $\beta_1 \in
B'$, then $\beta_1 \bullet \zeta=0$ by Proposition \ref{orthoprod} and
$\beta_1 \bullet \sigma = \beta_1 \bullet \sum_{j=0}^d \gamma_j =
\beta_1 \bullet \gamma_i$ for reasons of dimension.  Thus, $\beta_1=
\beta_1 \bullet [\Delta_Y]  = \beta_1 \bullet (\zeta + \sigma) =
\beta_1  \bullet \gamma_i = \beta_1 \bullet (\eta_i + \theta_i)= \beta_1 + \beta_1 \bullet \theta_i$.
It follows that $\beta_1 \bullet \theta_i=0$ in this case also.   \qed

\endproof

\begin{corollary}\label{ideal}
$B$ is a two-sided ideal of $\rch^*(Y \times Y)$.  
\end{corollary}

\proof

We need to check that for all $\alpha \in A$, and $\beta_1, \beta_2 \in B$, the elements $\alpha \bullet \beta_1$, $\beta_2 \bullet \alpha$, and $\beta_1  \bullet \beta_2$ are in $B$.  For the first, simply note that because $\zeta \in A$, $\zeta \bullet (\alpha \bullet \beta_1) \bullet \zeta = \zeta \bullet \alpha \bullet (\beta_1 \bullet \zeta) = 0$ by Proposition \ref{weakor}.  By Corollary \ref{abtest}, $\alpha \bullet \beta_1 \in B$.  The argument showing $\beta_2 \bullet \alpha \in B$ is similar.  

\medskip

We will show that if $\beta_1, \beta_2 \in B$, then in fact $\beta_1 \bullet \beta_2=0$, which is clearly in $B$.  By linearity in each factor, we may assume that $\beta_1, \beta_2$ are homogeneous elements with respect to the grading on $\rch^*(Y \times Y)$, i.e. $\beta_1 \in \rch^k(Y \times Y)$ and $\beta_2 \in \rch^l(Y \times Y)$ for some $k,l$.  For $i=1, 2$, write $\beta_i = \beta_i' + \beta_i''$, where $\beta_i \in B'$ and $\beta_i'' \in B''$.  Then 
$$\zeta \bullet (\beta_1 \bullet \beta_2)  \bullet \zeta = \zeta \bullet \beta_1' \bullet \beta_2'  \bullet \zeta + \zeta \bullet \beta_1'' \bullet \beta_2'  \bullet \zeta + \zeta \bullet \beta_1' \bullet \beta_2''  \bullet \zeta + \zeta \bullet \beta_1'' \bullet \beta_2''  \bullet \zeta.$$

The first, second, and fourth terms are zero by Proposition \ref{weakor}, so we may assume without loss of generality that 
$\beta_1 \in B'$ and $\beta_2 \in B''$.  By the projective bundle formula, $\beta_1$ is a sum of elements of 
the form $\ell_i \times b_{k-i}$ and $\beta_2$ is a sum of elements of the form $c_{l-j} \times \ell_j$, 
where $1 \leq i, j \leq d$ and $b_m, c_n \in \rch^*(Y)$.  Using linearity again, we reduce to the 
case  $\beta_1=\ell_i \times b_{k-i}$, and $\beta_2 = c_{l-j} \times \ell_j$.  Then
 $\beta_1 \bullet \beta_2 = (\beta_1 \bullet \eta_i) \bullet (\theta_{d-j} \bullet \beta_2)$ by Corollary \ref{onesided}.  
If $i=d-j$, then $\eta_i \bullet \theta_{d-j}=0$ by Proposition \ref{mvalues}.  If $i \neq d-j$, 
then $\eta_i \bullet \theta_{d-j} = \eta_i \bullet \gamma_i \bullet \gamma_{d-j} \bullet \theta_{d-j} = 0$ where the first equality is 
by ~\eqref{key.rels.1} and the second equality is by the orthogonality of $\gamma_i$ and $\gamma_{d-j}$ as shown in Lemma ~\ref{orthoprod}. 
In either case, we have $\beta_1 \bullet \beta_2 = 0.$ \qed

\endproof

\bigskip

\subsection{Blowing down}

We now have the tools to prove the converse of Theorem \ref{blowup}.

\medskip

By Proposition \ref{dstrans}, we may write every $\delta \in \rch^i(Y \times Y)$ uniquely as $(f \times f)^* (f \times f)_* \delta + b_{\delta}$, where $b_{\delta} \in B$.  As we will need to use a similar argument in Section \ref{lefschetz}, we phrase the next result in somewhat general form.

\begin{lemma}\label{drop}
With hypotheses as in Theorem \ref{blowup}, suppose $\delta_i \in \rch^*(Y \times Y)$, $1 \leq i \leq 4$, satisfy $\delta_1 \bullet \delta_2=\delta_3 \bullet \delta_4$.  Then $(f \times f)_* \delta_1 \bullet (f \times f)_* \delta_2 = (f \times f)_* \delta_3 \bullet (f \times f)_* \delta_4.$
\end{lemma}

\proof

For each $i=1, \ldots, 4$, write $\delta_i = (f \times f)^* (f \times f)_* \delta_i + b_i$, where $b_i \in B$.  Substituting these expression into the assumption $\delta_1 \bullet \delta_2 = \delta_3 \bullet \delta_4$ yields
$$(f \times f)^* (f \times f)_* \delta_1 \bullet (f \times f)^* (f \times f)_* \delta_2 + (f \times f)^* (f \times f)_* \delta_1 \bullet b_2 + b_1 \bullet (f \times f)^* (f \times f)_* \delta_2 + b_1 \bullet b_2$$
$$= (f \times f)^* (f \times f)_* \delta_3 \bullet (f \times f)^* (f \times f)_* \delta_3 + (f \times f)^* (f \times f)_* \delta_3 \bullet b_4 + b_3 \bullet (f \times f)^* (f \times f)_* \delta_4 + b_3 \bullet b_4,$$

which by Lemma \ref{compat} may be rewritten 
$$(f \times f)^* ((f \times f)_* \delta_1 \bullet (f \times f)_* \delta_2) + (f \times f)^* (f \times f)_* \delta_1 \bullet b_2 + b_1 \bullet (f \times f)^* (f \times f)_* \delta_2 + b_1 \bullet b_2$$
$$= (f \times f)^* ((f \times f)_* \delta_3 \bullet (f \times f)_* \delta_3) + (f \times f)^* (f \times f)_* \delta_3 \bullet b_4 + b_3 \bullet (f \times f)^* (f \times f)_* \delta_4 + b_3 \bullet b_4.$$

The first summand on each side is in $A$, while Corollary \ref{ideal}
shows that the other three summands are in $B$.  By Proposition
\ref{dstrans}, $\rch^*(Y \times Y)$ is the internal direct sum of $A$
and $B$, so we must have $(f \times f)^*((f \times f)_* \delta_1
\bullet (f \times f)_* \delta_2) = (f \times f)^* ((f \times f)_*
\delta_3 \bullet  (f \times f)_* \delta_4).$  Finally, the projection
formula implies that $(f \times f)_* (f \times f)^*$ is the identity
map, so $(f \times f)^*$ is injective and the assertion follows.   

\qed
\endproof

By taking $\delta_4=[\Delta_Y]$, we immediately deduce:

\begin{corollary}\label{pushdown}
If $\delta_1, \delta_2 \in \rch^*(Y \times Y)$, then $(f \times f)_*(\delta_1 \bullet \delta_2)=(f \times f)_* \delta_1 \bullet (f \times f)_* \delta_2$.  In particular, if $\delta_1 \bullet \delta_2=0$, then $(f \times f)_* \delta_1 \bullet (f \times f)_* \delta_2 = 0$.  
\end{corollary}

As discussed in \cite[1.4]{sv}, given a surjective morphism $g: V \arr W$ of projective varieties, one may 
identify the (Chow) motive of $W$ with a direct summand of the Chow motive of $V$.  In particular, there is a section $s \in \rch^d(W \times V)$ such that $[\Gamma_g]  \bullet s = [\Delta_W]$.  If one begins with a 
Chow-\kuns ~ decomposition $[\Delta_V] = \sum_{j=0}^{2d} \pi_j^V$ for $V$, one might attempt to construct a Chow-\kuns decomposition on $W$ by 
considering the elements $[\Gamma_g] \bullet \pi_j^V \bullet s$, $0 \leq j \leq 2d$.  Because the $\pi_j^V$ are central modulo homological equivalence (assuming some choice of Weil cohomology theory), the cohomology classes of the elements $\ds [\Gamma_g] \bullet \pi_j^Y \bullet s,  0 \leq j \leq 2d$ actually define a {\em \kun} decomposition on $W$, but it does not follow that these elements are idempotents when considered with respect to rational equivalence.  

\smallskip

In the case of blowing down in our special case as considered in ~\ref{key.assumpns}, however, this construction actually yields a Chow-\kuns decomposition as shown in the
following Corollary.

\begin{corollary}\label{blowdown}
With hypotheses as in Section \ref{maincon}, suppose
$[\Delta_Y]=\sum_{j=0}^{2d} \pi_j^Y$ is a Chow-\kun ~ 
decomposition for $Y$.  Then $[\Delta_X]=\sum_{j=0}^{2d} (f \times f)_* 
\pi_j^Y$ is a CK-decomposition for $X$.  Moreover, if the 
CK-decomposition for $Y$ satisfies Poincar\'{e} duality, then so does 
the CK-decomposition for $X$.
\end{corollary}

\begin{proof}

The formula for $[\Delta_X]$ follows by applying $(f \times f)_*$ to
the expression for $[\Delta_Y]$, noting that $f \times f$ has
degree $1$.  Now if $i \neq j$, we have $\pi_i^Y \bullet \pi_j^Y=0$, so
$(f \times f)_* \pi_i^Y \bullet (f \times f)_* \pi_j^Y=0$ by Corollary \ref{pushdown}.
Finally, $$(f \times f)_* \pi_i^Y  \bullet (f \times f)_* 
\pi_i^Y = ([\Delta_X] - \sum_{j \neq i} (f \times f)_* \pi_j^Y) \bullet (f \times f)_* \pi_i^Y = (f \times f)_* \pi_i^Y.$$
The remaining assertions are clear from the construction. \qed
\end{proof}

\section{\bf Application to Murre's Conjectures}\label{app1}

The goal of this section is to prove that each of Murre's Conjectures holds for $Y$ if and only if it holds for $X$.  The case of Murre's Conjecture \textbf{A} (existence of a Chow-\kun ~ 
decomposition) was completed in the previous section.  In the interest of 
making the proofs easier to follow, we use Greek letters for elements of 
$\rch^*(Y \times Y)$ or $\rch^*(X \times X)$ and Roman letters for elements of 
$\rch^*(Y)$ or $\rch^*(X)$.  In order to study the action of correspondences on Chow groups, we will need some results analogous to Lemma \ref{compat} and Proposition \ref{weakor}.

\begin{lemma}\label{smallsurprise}
If $\alpha \in \rch^*(X \times X)$ and $x\in \rch^*(X)$, then $$f^*(\alpha 
\bullet x) = (f \times f)^*(\alpha) \bullet f^*x.$$  

\end{lemma}

\proof

Viewing $x$ as an element of $\corr(k,X)$, we compute using Lemma \ref{dict}.  Keeping in mind that $[\Gamma_f] \bullet [\Gamma_f^t]$ is multiplication by $\deg ~ f = 1$, we have:

$$(f \times f)^* \alpha \bullet f^* x = [\Gamma_f^t] \bullet \alpha \bullet [\Gamma_f] \bullet [\Gamma_f^t] \bullet x = 
[\Gamma_f^t] \bullet \alpha \bullet x  = f^*(\alpha \bullet x).$$ \qed

\endproof

Recall that $T$ denotes the center of the blow-up of $X$ and that $Z$ is the exceptional divisor in $Y$. 
For $i$, $0 \leq i \leq d$, define subgroups $C_i= f^* \rch^i(X)$ and $D_i = 
j_*(\ker g_*: \rch^{i-1}(Z) \arr \rch^{i-d}(T))$ of $\rch^i(Y)$.  Observe that 
$D_i=\rch^{i-1}(Z) \cong \ratls$ if $i < d$ and $D_i=0$ if $i=d$.
Then 

\begin{proposition}\label{singleorth}
$\rch^i(Y)$ is the internal direct sum of $C_i$ and $D_i$.  Furthermore, if 
$\alpha \in A$ and $d_i \in D_i$, then $\alpha \bullet d_i=0$, and if $\beta \in B$ and $c \in C_i$, then $(\beta \bullet \sigma) \bullet c_i=0$.  
\end{proposition}

\begin{proof}

Let $V_X=X-T$ and $V_Y=Y - Z$.  Then $V_X \cong V_Y$, so localization gives a commutative diagram with exact rows:

\xymatrix{ 
& \rch^{i-1}(Z) \ar[rr]^{j_*} \ar[dd]^{g_*} & & \rch^i(Y) \ar[rr] 
\ar[dd]^{f_*} & & \rch^i(V_Y) \ar[dd]^{\cong} \ar[r] & 0  \\ \\
& \rch^{i-d}(T) \ar[rr]^{i_*} & & \rch^i(X) \ar[rr] & & \rch^i(V_X) 
\ar[r] & 0  
} 

The property $\rch^i(Y) = C_i+D_i$ follows from a straightforward diagram 
chase and the fact that $i_*$ is injective.  Now write $\alpha=(f \times f)^* u$, $c_i = f^* v \in C_i$, and 
$d_i=j_* y \in D_i$.  Then 
$$\alpha \bullet d_i = (f \times f)^* u \bullet j_* y = [\Gamma_f^t] \bullet u \bullet [\Gamma_f] \bullet [\Gamma_j] \bullet y
=[\Gamma_f^t] \bullet u \bullet [\Gamma_{f \bullet j}] \bullet y  =[\Gamma_f^t] \bullet u \bullet [\Gamma_{i \bullet g}] \bullet y  =[\Gamma_f^t] \bullet u \bullet [\Gamma_i] \bullet g_* y = 0.$$ 

Therefore, if $c_i\in C_i \cap D_i$, then (regarding $c_i$ as an element of $C_i$), we have $\zeta \bullet c_i=c_i$ by Lemma \ref{smallsurprise}, but also (regarding $c_i$ an element of $D_i$), $\zeta \bullet c_i=0$ by the above calculation.  Thus, $C_i \cap D_i=\{0\}$.  

Finally, 
$$(\beta \bullet \sigma) \bullet c_i = \beta \bullet \sigma \bullet f^*v =
 \beta \bullet \sigma \bullet \zeta \bullet f^*v = 
\beta \bullet 0 \bullet f^* v =0.$$

\qed \end{proof}

From this point onward, we fix identifications $$\rch^i(Y) \cong C_i \oplus 
D_i \mbox{ and } \rch^d(Y \times Y) \cong A \oplus B$$ and thereby justify the use of ordered pair notation for elements of $\rch^*(Y)$ and $\rch^*(Y \times Y)$.  

\begin{corollary}\label{distrib}
Suppose $((f \times f)^*\alpha, \beta) \in \rch^d(Y \times Y)$ and $(f^*x, y) \in 
\rch^i(Y)$.  Then $$((f \times f)^*\alpha, \beta \bullet \sigma)
\bullet (f^*x, y) = (f^*(\alpha \bullet x), \beta \bullet \sigma
\bullet y)$$ 
\end{corollary} 
 
\begin{proposition}\label{conjb}
Murre's Conjecture \textbf{B} holds for $X$ if and only if it holds for $Y$, and similarly for Conjecture \textbf{B}'.  
\end{proposition}

\begin{proof}

We give the proof for Conjecture \textbf{B}.  First suppose $X$ has a Chow-\kun ~ decomposition \\ $[\Delta_X] = \sum_{i=0}^{2d} 
\pi_i^X$  satisfying Murre's Conjecture \textbf{B}, i.e. $\pi_{\ell} \bullet \rch^j(X)=0$ 
when $\ell<j$ or $\ell>2j$, and let $[\Delta_Y]=\sum_{i=0}^{2d} ((f
\times f)^* \pi_i^X, \delta_i)$ be the Chow-\kun ~ decomposition for $Y$
as constructed in (\ref{CK.for.Y}).
From (\ref{gammass}), the property $\delta_i = \delta_i \bullet
\sigma$ holds for all $i$.  Now  fix $j$, $0 \leq j \leq d$, and
consider $(f^* x, y ) \in\rch^j(Y)$.  By Corollary \ref{distrib}, $((f
\times f)^* \pi_{\ell}^X, {\delta}_{\ell}) \bullet (f^* x, y) =
(f^*(\pi_{\ell}^X \bullet x), {\delta}_{\ell} \bullet y)$.  If $\ell <
j$ or $\ell >2j$, then $\pi_{\ell}^X \bullet x=0$.  When $\ell$ is odd,
clearly ${\delta}_{\ell} \bullet y=0$; so assume $\ell$ is even.  Then
${\delta}_{\ell}$ is a product cycle of type $(d-\ell/2, \ell/2)$; so
it suffices to show that for any $u \in {\rch}^{d-\ell/2}(Y)$ and $v
\in {\rch}^{\ell/2}(Y)$, $(u \times v) \bullet y=0$ when $\ell< j$ or
$\ell>2j$.  Then  

$$(u \times v) \bullet y = {p_2^{YY}}_* ({p_1^{YY}}^* y \cdot {p_1^{YY}}^* u \cdot {p_2^{YY}}^* v)= {p_2^{YY}}_* {p_1^{YY}}^*(y \cdot u) \cdot v={p_{\emptyset}^Y}^* {p_{\emptyset}^Y}_* (y \cdot u) \cdot v.$$

Note that $y \cdot u\in \rch^{j+d-\ell/2}(Y)$.  If $\ell < j$, then $j 
+ d - \ell /2 > d$; so $y \cdot u= 0$.  If $\ell > 2j$, then $j- \ell/2 
< 0$; so ${p_{\emptyset}^Y}_* (y \cdot u) \in \rch^{j-\ell/2}(\spec k) = 0$.  
Thus, this Chow-\kun ~ decomposition for $Y$ satisfies Murre's Conjecture 
\textbf{B}.  

\medskip

Conversely, suppose $[\Delta_Y] = \sum_{i=0}^{2d} \pi_i^Y$ is a Chow-\kun ~ decomposition for $Y$ satisfying Murre's Conjecture
\textbf{B}.  By Proposition \ref{dstrans}, we may write $\pi_i^Y = ((f
\times f)^* \pi_i^X, \delta_i)$ for some $\pi_i^X \in \rch^d(X \times
X)$.  This means, in particular, that if $(f^*x, y) \in
\rch^{\ell}(Y)$, then $(f \times f)^*\pi_j^X 
\bullet f^* x = 0$ when $\ell < j$  or $\ell > 2j$.  By Lemma 
\ref{smallsurprise} we have $f^*(\pi_j^X \bullet x)=0$, and since $f^*$ is injective, 
$\pi_j^X \bullet x=0$.  Corollary \ref{blowdown} then guarantees that
$[\Delta_X] = \sum_{i=0}^{2d} \pi_i^X$ is a Chow-\kun ~ decomposition
for $X$ satisfying Murre's Conjecture \textbf{B}.

\qed \end{proof}

\begin{proposition}
Murre's Conjecture \textbf{C} holds for $X$ if and only if it holds for $Y$.  
Similarly, Conjecture \textbf{D} holds for $X$ if and only if it holds for $Y$.
\end{proposition}

\begin{proof}

Assume first that $X$ has a CK-decomposition $[\Delta_X] =
\sum_{i=0}^{2d} \pi_i^X$ satisfying Murre's 
Conjecture \textbf{C}.  Now let $[\Delta_Y]= \sum_{i=0}^{2d} (f \times f)^* 
\pi_i^X + \delta_i$ be the CK-decomposition for $Y$
constructed in the proof of Theorem \ref{blowup}.  By Proposition \ref{singleorth} and Corollary 
\ref{distrib}, we have $((f \times f)^*\pi_{\ell}^X + \delta_{\ell}) \bullet 
\rch^i (Y) =  (f \times f)^*\pi_{\ell}^X \bullet 
C_i + \delta_{\ell} \bullet D_i= f^*(\pi_{\ell}^X \bullet 
C_i) + \delta_{\ell} \bullet D_i$.  In particular, this implies 
that the filtration induced by this CK-decomposition (as defined in
Section \ref{prelim1}) is described by 

\begin{eqnarray}\label{filt}
F^m\rch^i(Y) =f^* F^m \rch^i(X) + D_{i,m}
\end{eqnarray}

where $D_i= D_{i,0} \supseteq D_{i,1} \supseteq \ldots $ is a descending chain 
of 
subgroups.  By Murre's Conjecture \textbf{C} for $X$, the term $F^m \rch^i(X)$ 
is independent of the original choice of CK-decomposition for $X$.  Also, by 
\cite[Lemma 1.4.4]{mur}, $F^1 \rch^i(Y)$ is contained in the subgroup 
$\rch^i(Y)_{hom} \subseteq \rch^i(Y)$ of cycles homologically equivalent to zero.  If $i=d$, 
then $D_i=0$; so $D_{i,j}=0$ for all $j$.  If $i<d$, then  $D_i=\rch^{i-1}(Z) 
\cong \ratls$ is a one-dimensional $\ratls$-vector space, with $\rch^{i-1}(Z)_{hom}=0$.  Hence $D_{i,j}=0$ for 
all $j \geq 1$, showing that the filtration $F^m \rch^i(Y)$ is independent of the 
original choice of CK-decomposition on $Y$.  This proof also shows 
that if Conjecture \textbf{D} holds for $X$, i.e. $F^1 \rch^i(X) = 
{\rch^i(X)}_{hom}$, then likewise $F^1 \rch^i(Y) = {\rch^i(Y)}_{hom}$.

\medskip

Conversely, suppose $Y$ has a CK-decomposition satisfying Murre's 
Conjecture \textbf{C}.    
If $[\Delta_X]=\sum_{i=0}^{2d} \pi_i^X$ is a CK-decomposition on $X$, use 
Theorem \ref{blowup} to construct a CK-decomposition $[\Delta_Y]= 
\sum_{i=0}^{2d} (f \times f)^* \pi_i^X + \delta_i$ on $Y$.  By assumption, the 
filtration (\ref{filt}) defined by this CK-decomposition is independent of the original choice 
of CK-decomposition on $X$; hence $f_*F^m\rch^i(Y) = F^m \rch^i(X)$ is 
also independent on this choice, and so Conjecture \textbf{C} holds
for $X$.  The assertion concerning Conjecture \textbf{D} follows
similarly.  

\qed \end{proof}

\section{\bf Application to Lefschetz decompositions}\label{lefschetz}

\subsection{Lefschetz type and blowing up}

The goal of this section is to prove the following result:

\begin{theorem}\label{leftheorem}
Let $f: Y \arr X$ be as in Section \ref{maincon}.  Then $Y$ is of
Lefschetz type if and only if $X$ is of Lefschetz type.
\end{theorem}

Let $d=\dim X=\dim Y$ and suppose $((\rch^{*+d}(X \times X),
\{\pi_i^X\}_{i=0}^{2d}, L_X, \Lambda_X)$ is a Lefschetz algebra.  It
follows easily from Lemma \ref{compat} that $$((f \times f)^* \rch^*(X
\times X),    \{(f \times f)^*\pi_i^X\}_{i=0}^{2d}  , (f \times f)^*
L_X, (f \times f)^* \Lambda_X)$$ is a Lefschetz algebra. 

Next, define

$$L' = \sum_{i=1}^{d} (\ell_i \times a_{d-i+1}) + (a_i \times \ell_{d-i+1}) \in 
\rch^{d+1}(Y \times Y)$$

and 

$$\Lambda' = \sum_{j=0}^{d-1} (d-j)(j+1) \left[ (\ell_j \times a_{d-j-1}) + 
(a_j \times \ell_{d-j-1}) \right] \in \rch^{d-1}(Y \times Y).$$

We will show that if we set $$L_Y= (f \times f)^* L_X + L' \in \rch^{d+1}(Y \times Y) \mbox{ and } \Lambda_Y = (f \times f)^* \Lambda_X + \Lambda' \in \rch^{d-1}(Y \times Y)$$
and denote by $\pi_0^Y, \ldots, \pi_{2d}^Y \in \rch^d(Y \times Y)$ the
projectors constructed in (\ref{CK.for.Y}), then
\\ $(\rch^{*+d}(Y \times Y), \{\pi_i^Y\}_{i=0}^{2d}, L_Y, \Lambda_Y)$ is
a Lefschetz algebra.

\begin{proposition}\label{primelefalg}
The following identities hold for each $s$:

\begin{itemize}

\item
$L' \bullet \delta_{2s} = \delta_{2s+2} \bullet L' =\ell_{d-s} \times a_{s+1} + a_{d-s} \times \ell_{s+1}. $

\item
$\Lambda' \bullet \delta_{2s} = \delta_{2s-2} \bullet \Lambda' = s(d-s+1)[(\ell_{d-s} \times a_{s-1}) + (a_{d-s} \times \ell_{s-1})]$.

\item
$\Lambda' \bullet L' - L' \bullet \Lambda' = \sum_{i=0}^{2d} (d-i) \delta_i $.  
\end{itemize}
\end{proposition}

\medskip

\proof

$$L' \bullet \delta_{2s} = L' \bullet \gamma_{d-s} = \sum_{i=1}^d [\ell_i \times a_{d-i+1} + a_i \times \ell_{d-i+1}] \bullet [\eta_{d-s} + \theta_{d-s}].$$

By Lemma \ref{orthoprod}, the only (possibly) nonzero term in this sum
corresponds to $i=d-s$, so
$$L' \bullet \delta_{2s} =  (\ell_{d-s} \times a_{s+1}) \bullet
\eta_{d-s} + (\ell_{d-s} \times a_{s+1}) \bullet \theta_{d-s} 
+ (a_{d-s} \times \ell_{s+1}) \bullet
\eta_{d-s} + (a_{d-s} \times \ell_{s+1}) \bullet \theta_{d-s}.$$

Using Proposition \ref{mvalues} and Corollary \ref{onesided}, we see that the first term equals $\ell_{d-s} \times
a_{s+1}$ and the fourth term equals $a_{d-s} \times
\ell_{s+1}$, while the middle two terms vanish.  Thus, $L' \bullet \delta_{2s} = \ell_{d-s} \times a_{s+1} +
a_{d-s} \times \ell_{s+1}$.  By similar reasoning, 
$$\delta_{2s+2} \bullet L' = (\eta_{d-s-1} + \theta_{d-s-1}) \bullet 
[\ell_{d-s} \times a_{s+1} + a_{d-s} \times \ell_{s+1}]= \ell_{d-s}
\times a_{s+1} + a_{d-s} \times \ell_{s+1}.$$  Therefore, $L' \bullet
\delta_{2s} = \delta_{2s+2} \bullet L'$, as desired.  The proof of the
second formula is similar.  

\medskip

For the third formula,

$$\Lambda' \bullet L' = \sum_{j=0}^{d-1} \sum_{i=1}^d (d-j)(j+1) \left[ (\ell_j \times a_{d-j-1}) + (a_j \times \ell_{d-j-1}) \right] \bullet \left[ (\ell_i \times a_{d-i+1}) + (a_i \times \ell_{d-i+1})  \right]$$

By Lemma \ref{orthoprod}, the $(i,j)$ term of this double sum
will be zero unless $j=i-1$; hence the expression simplifies to

$$\sum_{i=1}^d i (d-i+1) \left[ (\ell_{i-1} \times a_{d-i}) + (a_{i-1} \times \ell_{d-i}) \right] \bullet \left[ (\ell_i \times a_{d-i+1}) + (a_i \times \ell_{d-i+1})  \right]$$
$$= \sum_{i=1}^d i (d-i+1) [ m(\ell_{i-1}, a_{d-i+1}) (\ell_i \times a_{d-i}) + 
m(\ell_{i-1}, 
\ell_{d-i+1}) (a_i \times a_{d-i})$$  $$+ m(a_{i-1}, a_{d-i+1})(\ell_i \times \ell_{d-i}) + m(a_{i-1}, \ell_{d-i+1}) (a_i \times \ell_{d-i}) ]$$
$$=\sum_{i=1}^d i (d-i+1) [(\ell_i \times a_{d-i}) +  (a_i \times \ell_{d-i}) 
].$$
	
Likewise,
$$L' \bullet \Lambda' = \sum_{j=0}^{d-1} \sum_{i=1}^d (d-j)(j+1) 
\left[ (\ell_i \times a_{d-i+1}) + (a_i \times \ell_{d-i+1})  \right]
\bullet
\left[ (\ell_j \times a_{d-j-1}) + (a_j \times \ell_{d-j-1}) \right].  
$$

Again, the only nonzero terms correspond to the case $j=i-1$, so this simplifies 
to $$\sum_{i=1}^d i(d-i+1) 
\left[ (\ell_i \times a_{d-i+1}) + (a_i \times \ell_{d-i+1})  \right]
\bullet
\left[ (\ell_{i-1} \times a_{d-i}) + (a_{i-1} \times \ell_{d-i}) \right]  
$$

$$=\sum_{i=1}^d i(d-i+1) [m(\ell_i, a_{d-i}) (\ell_{i-1} \times a_{d-i+1}) + 
m(\ell_i, \ell_{d-i})(a_{i-1} \times a_{d-i+1})$$ 
$$+ m(a_i, a_{d-i}) 
(\ell_{i-1} 
\times \ell_{d-i+1}) + m(a_i, \ell_{d-i})(a_{i-1} \times \ell_{d-i+1})]$$
$$=\sum_{i=1}^d i(d-i+1)[(\ell_{i-1} \times a_{d-i+1}) + (a_{i-1} \times 
\ell_{d-i+1})]=\sum_{i=0}^{d-1} (i+1)(d-i) [(\ell_i \times a_{d-i}) + (a_i \times 
\ell_{d-i})].$$

Hence,
\begin{align*}
\begin{split}
& \Lambda' \bullet L' - L' \bullet \Lambda' \\ 
&=-d[(\ell_0 \times a_d) + (a_0 \times \ell_d)] + \sum_{i=1}^{d-1} 
[i(d-i+1)-(i+1)(d-i)] [(\ell_i \times a_{d-i}) + (a_i \times
\ell_{d-i})]  + d[(\ell_d \times a_0) + (a_d \times \ell_0)] \\
&=-d \gamma_0 + \sum_{i=1}^{d-1}(2i-d) \gamma_i + d \gamma_d=\sum_{i=0}^d (2i-d) \gamma_i=\sum_{i=0}^{d} (2i-d) \delta_{2(d-i)}=\sum_{i=0}^{d} (d-2i) \delta_{2i}. \\
\end{split}
\end{align*}

Since $\delta_j=0$ for $j$ odd, reindexing gives $\Lambda' \bullet L' - L'
\bullet \Lambda'=\sum_{j=0}^{2d} (d-j) \delta_j$, as desired. 

\qed

\begin{corollary}\label{sigmastable}
$$L' \bullet \sigma = L'=\sigma \bullet L', ~ \Lambda' \bullet \sigma
= \sigma \bullet \Lambda' = \Lambda', \mbox{ and } \delta_j \bullet
\sigma = \sigma \bullet \delta_j = \sigma \mbox{ for } j,~ 0 \leq j \leq 2d.$$

\end{corollary}

\proof

By the definition of $L'$ and Proposition \ref{primelefalg}, $$L' = \sum_{i=0}^d L' \bullet
\delta_{2(d-i)} = L' \bullet \sum_{i=0}^d \delta_{2(d-i)} = L' \bullet
\sigma \mbox{ and } L' = \sum_{i=0}^d \delta_{2i} \bullet L'=
(\sum_{i=0}^d \delta_{2i}) \bullet L' = \sigma \bullet L',$$ 

which establishes the first set of equalities; the proof for the second set is similar. 
The third set of equalities is a restatement of (\ref{gammass}).  
\qed
\endproof

\medskip

\textbf{Proof of Theorem \ref{leftheorem}}

Suppose, as above, that $\ds (\rch^{*+d}(X \times X), \{\pi_i^X\}_{i=0}^{2d}, L_X, \Lambda_X)$ is a
Lefschetz algebra.  Let $\ds \{\pi_i^Y\}_{i=0}^{2d}$ denote the Chow-\kuns decomposition defined by (\ref{CK.for.Y}); we will show that \\ $(\rch^{*+d}(Y \times Y), \{\pi_i^Y\}_{i=0}^{2d}, (f \times f)^*L_X + L', (f \times f)^*\Lambda_X + \Lambda')$ is a Lefschetz
algebra. The verification is purely formal, so we will show it only
for the first formula in Definition \ref{lefalg}; the rest follow by
analogous reasoning. First,
$$((f \times f)^* L_X  + L') \bullet ((f \times f)^* \pi_{2s}^X +
\delta_{2s}) = (f \times f)^* L_X \bullet (f \times f)^* \pi_{2s}^X + (f
\times f)^* L_X \bullet \delta_{2s} + L' \bullet (f \times f)^*
\pi_{2s}^X + L' \bullet \delta_{2s}.$$ 

By Lemma \ref{compat} and the hypothesis, the first term on the
right equals $$(f \times f)^*(L_X \bullet \pi_{2s}^X) = (f \times
f)^*(\pi_{2s+2}^X \bullet L_X) = (f \times f)^*(\pi_{2s+2}^X) \bullet (f
\times f)^* L_X.$$

Since $(f \times f)^* L_X \in A$
and $\delta_{2s} =\sigma \bullet \delta_{2s}$ by (\ref{gammass}), we
may compute the second term:  $(f \times f)^* L_X \bullet \delta_{2s}
= (f \times f)^* L_X \bullet \zeta \bullet \sigma \bullet
\delta_{2s}=0.$  Likewise, the third term is  
$L' \bullet (f \times f)^* \pi_{2s}^X = L' \bullet \sigma \bullet \zeta
\bullet (f \times f)^* \pi_{2s}^X = 0.$   The last term equals
$\delta_{2s+2} \bullet L'$ by Proposition \ref{primelefalg}.  Thus, we
have:

$$((f \times f)^* L_X  + L') \bullet ((f \times f)^* \pi_{2s}^X +
\delta_{2s}) = (f \times f)^*(\pi_{2s+2}^X) \bullet (f \times f)^* L_X
+ \delta_{2s+2} \bullet L'.$$

Similarly one shows
$$(f \times f)^*(\pi_{2s+2}^X) \bullet (f \times f)^* L_X
+ \delta_{2s+2} \bullet L' = ((f \times f)^*(\pi_{2s+2}^X) +
\delta_{2s+2}) \bullet ((f \times f)^* L_X + L'),$$
completing the argument.

\medskip

Conversely, suppose $(\rch^{*+d}(Y \times Y), \{\pi_i^Y\}_{i=0}^{2d},
L_Y, \Lambda_Y)$ is a Lefschetz algebra, i.e. the identities of
Definition \ref{lefalg} are satisfied.  By Corollary  \ref{pushdown},
the analogous identities required to show that \\$(\rch^{*+d}(X \times
X), \{(f \times f)_*\pi_i^Y \}_{i=0}^{2d}, (f \times f)_* L_Y, (f \times
f)_* \Lambda_Y)$ is a Lefschetz algebra also hold.   
\qed

\endproof

\subsection{Agreement with the usual Lefschetz operator}

In the proof of Theorem \ref{leftheorem}, we constructed the Lefschetz operator $L_Y$ in terms of $L_X$ and an extra term $L'$.  
In the following we show that if $d=\dim X \geq 2$ and if $L_X$ takes the usual form of the Lefschetz operator, then the same is true for $Y$.  Letting $\Delta_X: X \arr X \times X$ and $\Delta_Y: Y \arr Y \times Y$ denote the respective diagonal maps, we show that if $L_X = (\Delta_X)_* (b)$ for some divisor $b \in \rch^1(X)$, then $L_Y = (\Delta_Y)_*(b')$ for some $b' \in \rch^1(Y)$.  

\medskip 


To this end, let $b \in \rch^1(X)$, and write $b'=f^*b + j_*[Z] \in \rch^1(Y)$.  Writing $L_{b'} = (\Delta_Y)_* (b')$, we have: $$(f \times f)_* L_{b'} = (f \times f)_* (\Delta_Y)_* (b') = (\Delta_X)_* f_* (b') = (\Delta_X)_* f_* f^*(b) +  (\Delta_X)_* f_* j_* [Z] = (\Delta_X)_* (b) +   (\Delta_X)_* i_* g_* [Z] = L_X.$$   

By Proposition \ref{dstrans}, there exists $M \in B_{d+1}$ such that $L_{b'} = \zeta \bullet L_{b'} \bullet \zeta + M$; that is, 

$$M = L_{b'} - \zeta \bullet L_{b'} \bullet \zeta = L_{b'} - (f \times f)^* (f \times f)_* L_{b'} = L_{b'} - (f \times f)^* L_X$$
We will show that $M=L'$; it will then follow that the Lefschetz operator $L_Y$ constructed in Theorem \ref{leftheorem} coincides with $L_{b'}$.  Note also that 

\begin{equation}\label{mexp}
M=L_{b'} - ([\Delta_Y]-\sigma) \bullet L_{b'} \bullet ([\Delta_Y]-\sigma) = \sigma \bullet L_{b'} + L_{b'} \bullet \sigma - \sigma \bullet L_{b'} \bullet \sigma
\end{equation}

\medskip

Observe that $$\sigma \bullet L_{b'} = \sum_{i=0}^d (a_i \times \ell_{d-i}  + \ell_i \times a_{d-i}) \bullet (\Delta_Y)_* (b')= \sum_{i=0}^d (a_i \bullet (\Delta_Y)_* (b')) \times \ell_{d-i}  + (\ell_i \bullet (\Delta_Y)_* (b')) \times a_{d-i}.$$

Now for $y \in \rch^i(Y)$, $y \bullet (\Delta_Y)_*(b') = {p_{2}^{YY}}_* ({p_1^{YY}}^* y \cdot (\Delta_Y)_* (b')) =  {p_{2}^{YY}}_*   (\Delta_Y)_* ((\Delta_Y)^* {p_1^{YY}}^* y \cdot b') = y \cdot b'$; thus,  $\ds \sigma \bullet L_{b'} = \sum_{i=0}^d  (a_i \cdot b') \times \ell_{d-i} + (\ell_i \cdot b') \times a_{d-i}.$ Also, $$\ell_i \cdot b' = j_* \ell^{i-1} \cdot (f^*(b) +   j_*[Z]) = j_*(\ell^{i-1} \cdot (j^* f^*(b) +  j^* j_* [Z])) = j_*(\ell^{i-1} \cdot g^* i^* (b)+  \ell^{i-1} \cdot  \ell) =  j_* \ell^i = \ell_{i+1}.$$

Thus, $\sigma \bullet L_{b'} = \sum_{i=0}^d  (a_i \cdot b') \times \ell_{d-i} +  \ell_{i+1} \times a_{d-i} .$  Now each term in the above summand is a product cycle of type $(i+1, d-i)$, so we have:
\begin{align*}
\begin{split}
& \sigma \bullet L_{b'} = (\sigma \bullet L_{b'}) \bullet \gamma_{i+1} \\
&=\sum_{i=0}^d ((a_i \cdot b') \times \ell_{d-i}) \bullet (a_{i+1} \times\ell_{d-i-1} + \ell_{i+1} \times a_{d-i-1}) +  (\ell_{i+1} \times a_{d-i}) \bullet 
(a_{i+1} \times \ell_{d-i-1} + \ell_{i+1} \times a_{d-i-1}) \\
&=\sum_{i=0}^d  m(\ell_{d-i-1}, a_i \cdot b') (a_{i+1} \times \ell_{d-i}) 
+ m(a_{d-i-1}, a_i \cdot b') (\ell_{i+1} \times \ell_{d-i})  \\ 
& +  m(\ell_{i+1}, \ell_{d-i-1}) (a_{i+1} \times a_{d-i}) +  m(\ell_{i+1}, a_{d-i-1}) (\ell_{i+1} \times a_{d-i}).
\end{split}
\end{align*}

By Proposition \ref{mvalues},the above simplifies to: $$\sigma \bullet L_{b'} = \sum_{i=0}^d  m(\ell_{d-i-1}, a_i \cdot b') (a_{i+1} \times \ell_{d-i}) 
+ m(a_{d-i-1}, a_i \cdot b') (\ell_{i+1} \times \ell_{d-i}) +
(\ell_{i+1} \times a_{d-i}).$$
Let $q: Y \arr \spec k$ be the structure morphism.  Then we have: 
$$m(\ell_{d-i-1}, a_i \cdot b') = q_* (j_* \ell^{d-i-2} \cdot  a_i \cdot (f^*(b) + r j_*[Z])) = q_*( j_*(\ell^{d-i-2} \cdot j^*(a_i) \cdot j^* f^*(b)) + q_*( j_*(\ell^{d-i-2} \cdot   j^*(a_i) \cdot j^* j_*[Z])$$ 
$$ =  q_*(j_*(\ell^{d-i-2} \cdot j^* (a_i) \cdot \ell)) =  q_* j_* (\ell^{d-i-1} \cdot j^*(a_i)) = q_* (j_* \ell^{d-i-1} \cdot a_i) =   m(\ell_{d-i}, a_i) =1.$$
Also, $m(a_{d-i-1}, a_i \cdot b')= q_*(a_{d-i-1} \cdot a_i \cdot b')$, but $a_{d-i-1} \cdot a_i = \Delta_Y^*(a_{d-i-1} \times a_i)=0$ by Proposition \ref{mvalues}, so $m(a_{d-i-1}, a_i \cdot b')=0$.

Therefore, $\sigma \bullet L_{b'} =  \sum_{i=0}^d   (a_{i+1} \times \ell_{d-i}) +  (\ell_{i+1} \times a_{d-i})=L'.$
From the definitions, one sees immediately that $\sigma^t = \sigma$, $L_{b'}^t = L_{b'}$, and the above calculation shows that $(\sigma \bullet L_{b'})^t = \sigma \bullet L_{b'}$.  Hence, $L_{b'} \bullet \sigma = L_{b'}^t \bullet \sigma^t = (\sigma \bullet L_{b'})^t = \sigma \bullet L_{b'}.$  By (\ref{mexp}), $M=\sigma \bullet L_{b'}=L'.$

\section{\bf Applications to Kummer varieties and manifolds}\label{app2}
Let $A$ be an abelian variety of dimension $d$ over an
algebraically closed field $k$ of characteristic $\neq 2$.  The
associated Kummer variety $K_A$ is obtained by taking the quotient of
$A$ by the action of the group (scheme) $G$ generated by the
involution $a \mapsto -a$.

\medskip

The Kummer manifold is obtained from $K_A$ by blowing up the singular locus of $A$ -- that is, by blowing up the image of the $2$-torsion points of $A$ under the quotient map $q: A \arr
K_A$.  As observed in \cite[p.4]{dl}, $K_A$ may be embedded in $\ds
{\mathbb{P}}^{2^d-1}$ using a symmetric theta divisor; thus, the image
of any $2$-torsion point is a singular point, \'{e}tale locally
isomorphic to the affine cone over the second Veronese variety of
${\mathbb{P}}^{d-1}$.  This can be seen by observing that the negation
involution of the abelian variety $A$ acts locally by
$(z_1,...,z_g)\arr (-z_1,\ldots,-z_g)$,  because it acts so on the
tangent space.  The ring of invariants is generated by polynomials
$z_iz_j$; hence, the exceptional divisor of the blow-up of the Kummer
variety at a $2$-torsion point is isomorphic to ${\mathbb P}^{d-1}$.
Now if $a \in A$ is a $2$-torsion point, $\overline{A}$ the blow-up of
$A$ along $\{a\}$, and $\overline{K}_A$ the blow-up of $K_A$ along
$\{q(a)\}$, then the universal property of the blow-up gives an
induced map $h: \overline{A}/G \arr \overline{K}_A$.  Since the exceptional
divisors of both blow-ups are (each) isomorphic to
${\mathbb{P}}^{d-1}$ and $K_A$ is known to be normal \cite{sas}, $h$ is a quasi-finite proper birational
map.  Because $\overline{K}_A$ is also normal, $h$ is an isomorphism by
Zariski's main theorem.  This proves that
$\overline{K}_A$ is also a pseudo-smooth variety.  A similar argument
shows that the intermediate schemes obtained by successively blowing up
each of the singular points on $K_A$ also satisfy the same
hypotheses.  Let $f: K_A' \arr K_A$ denote the composition of all these blow-up maps; $K_A'$ is then a smooth variety, 
the so-called Kummer manifold associated to $A$. 

\medskip

\subsection{Murre's conjectures and the Lefschetz decomposition for Kummer manifolds}
\begin{corollary}
The Kummer manifold $Y$ has a Chow-\kuns decomposition $[\Delta_{Y}] = \sum_{i=0}^{2d}
\pi_i^Y$ satisfying Poincar\'{e} duality, and is also of Lefschetz
type.  Furthermore, $Y$ satisfies Murre's conjecture \textbf{B'}, and when $d \leq 4$, $Y$ satisfies Murre's
conjecture \textbf{B}.  
\end{corollary}

\begin{proof}
 For convenience, let $X=K_A$ and $Y=K_A'$.   In \cite[Section 3]{dm}, Deninger and Murre constructed a particular 
Chow-\kuns decomposition $[\Delta_A]=\sum_{i=0}^{2d} \pi_i^A$ for $A$.   Using this construction, the present authors showed in \cite{aj1} that if we set $\pi_i^X=(q \times q)_* \pi_i^A \in \rch^d(X \times X)$, then  
$[\Delta_X]=\sum_{i=0}^{2d} \pi_i^X$ is a Chow-\kuns decomposition for $X$ satisfying Poincar\'{e}
duality and Murre's Conjecture \textbf{B}'; when $d \leq 4$, $X$ also satisfies  Murre's Conjecture \textbf{B}.  We also showed in \cite[Theorem 1.2]{aj2} that $X$ is of Lefschetz type.  The conclusion then follows by application of Theorem \ref{blowup}, Theorem \ref{leftheorem}, and Proposition \ref{conjb}. \qed
\end{proof}
\vskip .2cm \noindent


\textbf{Remarks.} 

\smallskip

1. A similar result holds for any of the intermediate
schemes obtained by blowing up some (but not all) of the singular points on $X$.
\vskip .2cm  \noindent
2. The referee has pointed out an alternate strategy for constructing an explicit Chow-\kuns decomposition on the Kummer manifold $Y$, based on a different construction of the latter. 
The negation map $a \mapsto -a$ on the abelian variety $A$  defines an action of $G=\zah/2 \zah$ on $A$ in the obvious manner.   If one blows up the locus of $2$-torsion points on $A$ to obtain a variety $\tilde{A}$, then the action of $G$ on $A$ extends in a natural way to an action of $G$ on $\tilde{A}$.  The Kummer manifold $Y$ can then be 
realized as the quotient variety $\tilde{A}/G$.  Therefore, if one starts with a Chow-\kuns decomposition on the Abelian variety that is stable under the action of $\zah/2\zah$, one could apply the results of \cite[Remark 5.5]{v} or \cite[Proposition 2.10]{sv} to obtain an explicit Chow-Kunneth decomposition on
$\tilde A$, which could then be descended to a Chow-Kunneth decomposition on the Kummer manifold $Y$ as in \cite[Corollary 2.13]{sv}.

\medskip

\subsection{Algebraic equivalence on Kummer varieties and manifolds}

Continuing the notation and assumptions of the previous section, we apply our explicit construction to study powers of the relation of algebraic equivalence on $X$ and on $Y$.

\medskip

In \cite{sam}, Samuel defined the notion of an {\em adequate equivalence
relation} on algebraic cycles and proved that rational equivalence is the finest such
relation.  Having fixed a field $k$, an adequate equivalence relation $E$ is an
assignment, to every pseudo-smooth variety $V$ over $k$, of a
subgroup $ECH^*(V) \subseteq CH^*(V)$, which is preserved under
pullback, pushforward, and intersection with arbitrary cycles.
Algebraic equivalence, homological equivalence, and numerical
equivalence all examples of adequate equivalence relations.  Hiroshi
Saito \cite{sai} defined the product $E*E'$ of two adequate equivalence
relations $E$ and $E'$, and proved that $E*E'$ is itself adequate.  He
also proved that the operation is associative and commutative, and
distributes in the expected manner over sums of relations (defined in the expected 
manner).

\medskip

If $A$ is an abelian variety of dimension $d$ over an algebraically closed field $k$, 
there is a natural filtration on $\rch^p(A)$ defined by the Deninger-Murre Chow-\kuns projectors: 
for $r \in \zah$, set $F^r \rch^p(A)=\sum_{i=0}^{2p-r} \pi_i^A \bullet \rch^p(A)$.  
A conjecture of Beauville \cite{be} is equivalent to the assertion that the nontrivial steps in 
this filtration occur only in positive degree, i.e. $F^0 \rch^p(A) = \rch^p(A)$.  This is easily seen to be 
equivalent to the assertion that $\pi_i$ acts as $0$ on $\rch^j(A)=0$ when $i<2j$, which is the second half of 
Murre's Conjecture \textbf{B}.

\medskip

Now let $L$ denote the (adequate) relation of algebraic equivalence.
Its $r$th power $L^{*r}$ is the so-called {\em $r$-cubical equivalence} introduced
in \cite{sam}.  In previous work of the first author, the following was proved
in a slightly stronger form:

\begin{theorem}\label{adeqr}\cite[Theorem 3.1 and Proposition 3.3]{ak}
Assume Beauville's Conjecture, and let $A$ be an abelian variety over an algebraically closed field.  Then: 

\begin{enumerate}
\item
For $r \geq 1$, $F^r \rch^d(A) = L^{*r} \rch^d(A)$.  

\item
For $r>d$, $L^{*r} \rch^*(A)=0$.

\end{enumerate}

\end{theorem}

The second statement is a kind of nilpotence assertion for cycles on abelian varieties.  We will show that our constructions yield similar results for $X$ and $Y$.

\medskip

First, define filtrations on $\rch^p(X)$ and $\rch^p(Y)$ by 

$$F^r \rch^p(X) = \sum_{i=0}^{2p-r} \pi_i^X \bullet \rch^p(X) \mbox{ and } F^r \rch^p(Y)=\sum_{i=0}^{2p-r} \pi_i^Y \bullet \rch^p(Y).$$

Then for $i$, $0 \leq i \leq 2d$ and $\alpha \in \rch^*(X)$, we have:
\begin{equation}\label{pushfil}
(q \times q)_* \pi_i^A \bullet \alpha = q_* (\pi_i^A \bullet q^* \alpha)
\end{equation}

since both sides are equal (as correspondences) to $\Gamma_q \bullet \pi_i^A \bullet \Gamma_q^t \bullet \alpha$.  Note that this identity is also expressed by the formula $F^r \rch^p(X) = q_* F^r \rch^p(A)$.  

\begin{proposition}\label{adeqkummer}
Assume Beauville's conjecture.  Then the conclusions of Theorem \ref{adeqr} hold when $A$ is replaced by either $X$ or $Y$.
\end{proposition}

\proof

Suppose $r \geq 1$.  Then Theorem \ref{adeqr}, together with adequacy of $L^{*r}$, implies
$$F^r \rch^d(X)=q_* F^r \rch^d(A) = q_* L^{*r} \rch^d(A) \subseteq L^{*r} \rch^d(X)$$

Likewise, since $q_* q^*$ is multiplication by $|G|$, 
$$L^{*r} \rch^d(X) = q_* q^* L^{*r} \rch^d(X) \subseteq q_*(L^{*r} \rch^d(A))=q_*(F^r \rch^d(A))=F^r \rch^p(X).$$

This proves the first statement for $X$.  For the second statement, simply observe that for $r>d$,

$$L^{*r} \rch^*(X) = q_* q^* L^{*r} \rch^*(X) \subseteq q_*(L^{*r} \rch^*(A))=0.$$

\medskip

To deduce the statements for $Y$, apply Proposition \ref{singleorth}  to write $\rch^d(Y) = C_d + D_d$.  Direct computation shows that $D_d=0$, so since $\pi_i^Y = (f \times f)^* \pi_i^X + \delta_i$ and $\delta_i = \delta_i \bullet \sigma$ by Corollary \ref{sigmastable}, another application of  Proposition \ref{singleorth} implies 
$$F^r \rch^d(Y)=\sum_{i=0}^{2d-r} \pi_i^Y \bullet \rch^d(Y) = 
\sum_{i=0}^{2d-r} (f \times f)^* \pi_i^X \bullet C_d  = \sum_{i=0}^{2p-r} f^*
(\pi_i^X \bullet \rch^d(X))   = f^* F^r \rch^d(X).$$
Now, using adequacy of $L^{*r}$, we have, for $r \geq 1$:
$$F^r \rch^d (Y) = f^* F^r \rch^d(X)=f^* L^{*r} \rch^d(X)
\subseteq L^{*r} \rch^d(Y).$$ 
Also, because $\rch^d(Y) = f^* \rch^d(X)$, we have $\rch^d(Y) = f^* f_* \rch^d(Y)$, so 
$$L^{*r} \rch^d(Y)  = f^* f_* L^{*r} \rch^d(Y) \subseteq f^* L^{*r} \rch^d(X)  = f^* F^r \rch^d(X)= F^r \rch^d(Y).$$
This establishes the first statement.  For the second, simply note that
for $r>d$, $$L^{*r} \rch^*(Y) = f^* f_* L^{*r} \rch^*(Y) \subseteq f^*
L^{*r} \rch^*(X) = 0.$$

\qed

\subsection{A Hard Lefschetz Theorem for Chow groups of Kummer manifolds}

As an application of the explicit Lefschetz decomposition constructed in Section \ref{lefschetz}, we prove the following theorem.  

\begin{theorem} (Hard Lefschetz for Chow groups)\label{harlef}
With notation and assumptions as in Sections \ref{maincon} and \ref{lefschetz}, suppose further that $X$ is of Lefschetz type and that for $2p \leq d$, the map $H_X: \rch^p(X) \arr \rch^{d-p}(X)$ defined by $a \mapsto L_X^{d-2p} \bullet a$ is an isomorphism.  Then the map $H_Y: \rch^p(Y) \arr \rch^{d-p}(Y)$ defined by $z \mapsto L_Y^{d-2p} \bullet z$ is an isomorphism.
\end{theorem}

\proof

By the direct sum decomposition $\rch^i(Y) \cong C_i\oplus D_i$ from Section \ref{app1}, any $z \in \rch^i(Y)$ may be written (uniquely) as $z=f^*x + y$, where $x \in \rch^i(X)$ and $y \in D_i$.  Then $$L_Y \bullet z = ((f \times f)^* L_X + L') \bullet (f^*x + y) = (f \times f)^* L_X \bullet f^* x + (f \times f)^* L_X \bullet y + L' \bullet f^*x  + L' \bullet y.$$

Using Proposition \ref{smallsurprise} to simplify the first two terms, and the equalities $\sigma \bullet L' = L' = L' \bullet \sigma$ from Corollary \ref{sigmastable}, we conclude: $L_Y \bullet z=f^*(L_X \bullet x) + (f \times f)^* L_X \bullet y  + L' \bullet \sigma \bullet f^*x+ L' \bullet y.$  
By Proposition \ref{singleorth}, the middle two terms are $0$, so we have $L_Y \bullet z = f^*(L_X \bullet x) + L' \bullet y.$  Using ordered pair notation (as in Section \ref{app1}) to express the decompositions $\rch^*(Y \times Y) \cong A \oplus B$ and $\rch^i(Y) \cong C_i \oplus D_i$, we have: 
$$L_Y \bullet z = ((f \times f)^*L_X , L') \bullet (f^*x, y) = (f^*(L_X \bullet x), L' \bullet y)$$
and hence, by induction,
$$L_Y^{d-2p} \bullet z=  (f^*(L_X^{d-2p} \bullet x), {L'}^{d-2p} \bullet y).$$
Thus, to prove that the map $H_Y: C_p \oplus D_p \arr C_{d-p} \oplus D_{d-p}$ defined above is an isomorphism, it suffices to check that the induced maps $u:C_p \arr C_{d-p}$ defined by $f^*x \mapsto f^*(L_X^{d-2p} \bullet x)$ and $v: D_p \arr D_{d-p}$ defined by $y \mapsto {L'}^{d-2p} \bullet y$ are isomorphisms.  That $u$ is an isomorphism follows formally, since $H_X$ is an isomorphism and $f^*$ is injective.

\medskip

We will also need an unweighted version of the $\Lambda'$ operator, defined by:
$$\Lambda_0' = \sum_{j=0}^{d-1} \ell_j \times a_{d-j-1} + a_j \times \ell_{d-j-1} \in \rch^{d-1}(Y \times Y).$$
We claim that the map $D_{d-p} \arr D_p$ defined by $y \mapsto \Lambda_0'^{d-2p} \bullet y$ is a two-sided inverse to $v$.  Fortunately, both $L'$ and $\Lambda_0'$ are product cycles, so we can calculate their powers explicitly.  By Lemma \ref{orthoprod}, we have: 

$${L'}^{2} = \sum_{i=1}^d [(\ell_i \times a_{d-i+1}) + (a_i \times \ell_{d-i+1})] \bullet \sum_{j=1}^d [(\ell_j \times a_{d-j+1}) + (a_j \times \ell_{d-j+1})]$$ $$= \sum_{i=1}^d [(\ell_i \times a_{d-i+1}) + (a_i \times \ell_{d-i+1})] \bullet  [(\ell_{i+1} \times a_{d-i}) + (a_{i+1} \times \ell_{d-i})]$$
$$=\sum_{i=1}^d  [ m(\ell_i, a_{d-i}) \ell_{i+1} \times a_{d-i+1} + m(\ell_i, \ell_{d-i}) a_{i+1} \times a_{d-i+1} + m(a_i, a_{d-i}) \ell_{i+1} \times \ell_{d-i+1} + m(a_i, \ell_{d-i}) a_{i+1} \times \ell_{d-i+1}$$

By Proposition \ref{mvalues}, the middle two terms vanish and the expression simplifies to $\sum_{i=1}^d   \ell_{i+1} \times a_{d-i+1} +  a_{i+1} \times \ell_{d-i+1} = \sum_{i=2}^d   \ell_{i} \times  a_{d-i+2} +  a_i \times \ell_{d-i+2}$.  Arguing inductively, we conclude

$${L'}^{d-2p} = \sum_{i=d-2p}^d   \ell_{i} \times  a_{2d-2p-i} +  a_i \times \ell_{2d-2p-i}.$$

Similarly, we compute

$$\Lambda_0'^{d-2p} = \sum_{j=0}^{2p} \ell_j \times a_{2p-j} + a_j \times \ell_{2p-j}.$$
If $z \in \rch^p(Y)$, direct computation shows that for $\alpha \in \rch^i(Y)$, $\beta \in \rch^j(Y)$, we have $(\alpha \times \beta) \bullet z=0$ unless $i+p=d$.  Using this principle, we see 
$\ds {L'}^{d-2p} \bullet z = [\ell_{d-p} \times a_{d-p} + a_{d-p}
\times \ell_{d-p}] \bullet z$ and hence 
$$\Lambda_0'^{d-2p} \bullet ({L'}^{d-2p} \bullet z) = [a_p \times \ell_p + \ell_p  \times a_p] \bullet  [\ell_{d-p} \times a_{d-p} + a_{d-p} \times \ell_{d-p}] \bullet z=(a_{d-p} \times \ell_p + \ell_{d-p} \times a_p) \bullet z=\gamma_{d-p} \bullet z.$$
Now suppose further that $z \in D_p \subseteq \rch^p(Y)$.  Then 
$$z=[\Delta_Y] \bullet z = (f \times f)^*[\Delta_X] \bullet z + \sigma \bullet z.$$
The first term vanishes by Proposition \ref{mvalues}, and since $\sigma = \sum_{i=0}^d \gamma_i$, where each $\gamma_i$ is a product cycle of type $(i,d-i)$, we see that $\sigma \bullet z= \gamma_{d-p} \bullet z$.  

Summarizing, we have 
$$(\Lambda_0'^{d-2p} \bullet {L'}^{d-2p}) \bullet z = z.$$
A similar calculation shows that for $e \in D_{d-p} \subseteq \rch^{d-p}(Y)$, 
$$({L'}^{d-2p} \bullet \Lambda_0'^{d-2p}) \bullet e = e.$$
This shows that the maps $z \mapsto {L'}^{d-2p} \bullet z$ and $e
\mapsto \Lambda_0'^{d-2p} \bullet e$ are mutually inverse isomorphisms between $D_p$
and $D_{d-p}$, completing the proof. \qed

\medskip

By the results of \cite{aj2}, the hypotheses of Theorem \ref{harlef} are satisfied for Kummer varieties over finite fields.  By taking direct limits, one easily argues that they also hold for Kummer varieties over the algebraic closure of a finite field.  Thus, we have:   

\begin{corollary}\label{hardlefth}
Let $Y$ be the Kummer manifold associated to an abelian variety of dimension $d>0$ over an  algebraic closure of some finite field of characteristic different from $2$.  Then for $2p \leq d$, the map $\rch^p(Y) \arr \rch^{d-p}(Y)$ defined by $z \mapsto L_Y^{d-2p} \bullet z$ is an isomorphism.  
\end{corollary}


\end{document}